\documentclass{amsart}
\usepackage{amssymb}
\usepackage{stmaryrd}
\usepackage{amsmath}
\usepackage{amscd}
\usepackage{amsthm}
\usepackage{amsbsy}
\usepackage[margin=1.2 5in]{geometry}
\usepackage{commath, longtable}
\usepackage{comment, enumerate}
\usepackage{geometry}
\usepackage[matrix,arrow]{xy}
\usepackage{hyperref}
\usepackage{mathrsfs}
\usepackage{color}
\usepackage{float}
\usepackage{mathtools,caption}
\usepackage[table,dvipsnames]{xcolor}
\usepackage{tikz-cd}
\usepackage{longtable}
\usepackage[utf8]{inputenc}
\usepackage[OT2,T1]{fontenc}
\usepackage[normalem]{ulem}
\usepackage{hyperref}
\usepackage[customcolors]{hf-tikz}
\usepackage{enumitem}
\newlist{primenumerate}{enumerate}{1}
\setlist[primenumerate,1]{label={\arabic*$'$}}
\hypersetup{
 colorlinks=true,
 linkcolor=DarkOrchid,
 filecolor=blue,
 citecolor=olive,
 urlcolor=orange,
 pdftitle={GHKL project},
 }
\usepackage{booktabs}

\DeclareSymbolFont{cyrletters}{OT2}{wncyr}{m}{n}
\DeclareMathSymbol{\Sha}{\mathalpha}{cyrletters}{"58}

\newcommand{\hh}{{\heartsuit\heartsuit}}
\newtheorem{theorem}{Theorem}[section]
\newtheorem{lemma}[theorem]{Lemma}

\newtheorem{proposition}[theorem]{Proposition}
\newtheorem{corollary}[theorem]{Corollary}
\newtheorem{definition}[theorem]{Definition}

\newtheorem*{notn}{Notation}
\numberwithin{equation}{section}
\newtheorem{lthm}{Theorem}

\theoremstyle{remark}
\newtheorem{remark}[theorem]{Remark}
\newtheorem{example}[theorem]{Example}

\newcommand{\EC}{\mathsf{E}}
\newcommand{\bb}{{\bullet}}
\newcommand{\Ep}{\EC[p^\infty]}
\newcommand{\fp}{\mathfrak{p}}
\newcommand{\fq}{\mathfrak{q}}
\newcommand{\Tr}{\operatorname{Tr}}
\newcommand{\Gal}{\operatorname{Gal}}

\newcommand{\PP}{\mathbb{P}}
\newcommand{\Qp}{\mathbb{Q}_p}
\newcommand{\Zp}{\mathbb{Z}_p}

\newcommand{\ord}{\mathrm{ord}}

\newcommand{\cD}{\mathcal{D}}

\newcommand{\Z}{\mathbb{Z}}
\newcommand{\Q}{\mathbb{Q}}
\newcommand{\F}{\mathbb{F}}

\newcommand{\cL}{\mathcal{L}}

\newcommand{\cK}{\mathcal{K}}

\newcommand{\cT}{\mathcal{T}}
\newcommand{\rk}{\mathrm{rank}}
\newcommand{\tors}{\mathrm{tors}}

\newcommand{\fr}{\mathfrak{r}}
\newcommand{\image}{\mathrm{Im}}

\newcommand{\Hom}{\mathrm{Hom}}

\newcommand{\Sel}{\mathrm{Sel}}
\newcommand{\Char}{\mathrm{char}}

\newcommand{\coker}{\mathrm{coker}}

\newcommand{\rank}{\mathrm{rank}}

\newcommand{\corank}{\mathrm{corank}}

\newcommand{\fP}{\mathfrak{P}}

\newcommand{\cyc}{\mathrm{cyc}}

\newcommand{\sM}{\mathscr{M}}
\newcommand{\sT}{\mathscr{T}}
\newcommand{\ac}{{\mathrm{ac}}}
\newcommand{\ann}{{\mathrm{ann}}}

\newcommand{\bc}{{\bullet\star}}

\makeatletter
\theoremstyle{plain} 
\newtheorem*{intr@thm}{\intr@thmname}

\newtheorem*{c@njecture}{\conjn@name}
\newcommand{\myl@bel}[2]{
 \protected@write \@auxout {}{\string \newlabel {#1}{{#2}{\thepage}{#2}{#1}{}} }
 \hypertarget{#1}{}
 } 
\newenvironment{labelledconj}[3][]
 {
 \def\conjn@name{#2}
 \begin{c@njecture}[{#1}]\myl@bel{#3}{#2}
 }
 {
 \end{c@njecture}
 }
\makeatother

\makeatletter
\newcommand{\mylabel}[2]{#2\def\@currentlabel{#2}\label{#1}}
\makeatother

\title[]{On a conjecture of Mazur predicting the growth of Mordell--Weil ranks in $\Zp$-extensions}

\author[R.~Gajek-Leonard]{Rylan Gajek-Leonard}
\address[Gajek-Leonard]{Department of Mathematics\\
Union College\\
Bailey Hall 108B\\
Schenectady, NY 12308\\
USA}
\email{gajekler@union.edu}

\author[J.~Hatley]{Jeffrey Hatley}
\address[Hatley]{
Department of Mathematics\\
Union College\\
Bailey Hall 202\\
Schenectady, NY 12308\\
USA}
\email{hatleyj@union.edu}

\author[D.~Kundu]{Debanjana Kundu}
\address[Kundu]{Department of Mathematical and Statistical Sciences\\ UTRGV \\ 1201 W University Dr.\\ Edinburg, TX 78539\\ USA}
\email{dkundu@math.toronto.edu}

\author[A.~Lei]{Antonio Lei}
\address[Lei]{Department of Mathematics and Statistics\\University of Ottawa\\
150 Louis-Pasteur Pvt\\
Ottawa, ON\\
Canada K1N 6N5}
\email{antonio.lei@uottawa.ca}


\begin{document}

\maketitle

\begin{abstract}
Let $p$ be an odd prime.
We study Mazur's conjecture on the growth of the Mordell--Weil ranks of an elliptic curve $E/\mathbb{Q}$ over $\Zp$-extensions of an imaginary quadratic field, where $p$ is a prime of good reduction for $E$.
In particular, we obtain criteria that may be checked through explicit calculation, thus allowing for the verification of Mazur's conjecture in specific examples.
\end{abstract}

\section{Introduction}
Let $\EC/K$ be an elliptic curve over a number field $K$ with good reduction at all primes above an odd prime number $p$.
By the Mordell--Weil theorem, for any \textit{finite} extension $L/K$, the group $\EC(L)$ of $L$-rational points is a finitely generated $\Z$-module, and thus has finite rank.
Upon replacing $L$ with an {\it infinite} algebraic extension of $K$, this may no longer be true.
In particular, if we consider a $\Zp$-extension of $K$, then the Mordell--Weil rank of $\EC$ may be unbounded in this infinite tower.

By the famous work of K.~Kato \cite{kato04} and D.~Rohrlich \cite{rohrlich88}, when $K=\Q$ and $\Q_\cyc$ is the unique cyclotomic $\Zp$-extension of $\Q$, the rank of $\EC(\Q_\cyc)$ indeed remains finite.
However, when $K/\Q$ is an imaginary quadratic field and $K_\ac$ is the corresponding anti-cyclotomic $\Zp$-extension, there are examples where $\EC(K_\ac)$ has infinite rank (see, for example, \cite[(1.9) or (1.10)]{Gre_PCMS}).

In \cite[\S~18, p.~201]{Maz84}, the following conjecture was made regarding the growth numbers in the \emph{generic} setting (see Remark~\ref{clarifying Mazur's conj}).

\begin{labelledconj}{Mazur's Growth Number Conjecture}{Mazur-conj} Suppose $\EC/\Q$ has good reduction at $p$.
The Mordell--Weil rank of $\EC$ stays bounded along any $\Zp$-extension of the imaginary quadratic field $K$, unless the extension is the anti-cyclotomic one and the root number of $\EC/K$ is $-1$.
\end{labelledconj}

The above conjecture was originally made by B.~Mazur when $p$ is a prime of good ordinary reduction.
This conjecture was later extended to supersingular primes (for example, in \cite{leisprung}).

For elliptic curves with complex multiplication (CM), the aforementioned conjecture is well understood from the works of K.~Rubin, D.~Rohrlich, B.~Gross-D.~Zagier, and R.~Greenberg (see for example \cite[Theorems~1.7 and 1.8]{Gre_PCMS}).
More precisely, let $\EC/\Q$ be an elliptic curve \textit{with complex multiplication} by an imaginary quadratic field $K$.
If $p$ is any prime, then the Mordell--Weil rank of $\EC$ remains bounded in every $\Zp$-extension which is different from $K_{\ac}$.
In the anti-cyclotomic $\Zp$-extension $K_{\ac}/K$, the Mordell--Weil rank of a CM elliptic curve is always unbounded if $p$ is a prime of good \textit{supersingular} reduction.
On the other hand, if $p$ is a prime of good \textit{ordinary} reduction then the boundedness of the Mordell--Weil rank of $\EC$ in $K_{\ac}/K$ depends on the parity of the order of vanishing of the Hasse-Weil $L$-function $L(\EC/\Q,s)$ at $s=1$.

Let $K$ be an imaginary quadratic field in which $p$ is unramified and write $K_\infty$ for the $\Zp^2$-extension of $K$.
When $p$ is a prime of good \emph{ordinary} reduction, it has been proven in \cite{KMS} that the number of $\Zp$-extensions of $K$ where the rank of $\EC$ does not stay bounded is at most $\min_H\{\lambda_H\}$, where $H$ runs over subgroups of $\Gal(K_\infty/K)$ such that $K_\infty^H$ is an admissible (in the sense of Definition~1.1 of \emph{op. cit.}) $\Zp$-extension of $K$ satisfying the $\mathfrak{M}_H(G)$-property, and $\lambda_H$ is the $\lambda$-invariant of the Pontryagin dual of the $p$-primary Selmer group of $E$ over that $\Zp$-extension.
In particular, the authors prove that \ref{Mazur-conj} holds when the Mordell--Weil rank of $\EC(K)$ is zero or one, under some hypotheses (see in particular Theorems 9.3 and 9.4 in \textit{op. cit.}).
For partial progress toward \ref{Mazur-conj} in the case of non-CM elliptic curves with supersingular reduction at $p$, we refer the reader to \cite{leisprung,hunglim}.

The goal of this article is to study this conjecture for elliptic curves with good reduction at $p$ (in both ordinary and supersingular settings).
In particular, we obtain criteria (see Corollaries~\ref{cor:main-result-ord} and \ref{cor:main-result-ss}) that are checked through explicit calculation (see \S~\ref{sec:SC hyp}), thus allowing for the verification of Mazur's conjecture in many specific examples, as illustrated in \S~\ref{sec:examples}.
The only unverifiable assumption is on the finiteness of the Shafarevich--Tate group.
All other hypotheses are provable or computable.

Our first main result is to prove sufficient conditions for \ref{Mazur-conj}.

\begin{lthm}
\label{Thm A}
Let $\EC/\Q$ be an elliptic curve of conductor $N_{\EC}$ and let $K$ be an imaginary quadratic field in which $p$ is unramified.
Let $p$ be an odd prime of good reduction of $\EC$.
Let $\Lambda$ denote the 2-variable Iwasawa algebra $\Zp\llbracket X,Y\rrbracket \simeq \Zp\llbracket \Gal(K_\infty/K)\rrbracket$ where $K_\infty/K$ is the unique $\Zp^2$-extension of $K$.
Further set $Y$ to denote the cyclotomic variable, i.e., $\Lambda_{\cyc}=\Zp\llbracket Y\rrbracket$.
In addition,
\begin{itemize}
\item if $p$ is a prime of good \emph{ordinary} reduction, suppose that
\begin{itemize}
\item  $a_v(\EC/K) \not\equiv 1 \pmod{p}$ for every $v \mid p$ in $K$.
\item $\Sel(\EC/K_\infty)^\vee$ is a direct sum of cyclic torsion $\Lambda$-modules.
\item The modified Heegner hypothesis\footnote{see \eqref{hyp: GHH} for the precise definition; the sign of the functional equation of $L(\EC/K,1)$ is $-1$ under this condition.} holds for the pair $(K, N_{\EC})$.
\item $\Sha(\EC/K_{\ac,(n)})[p^\infty]$ is finite for all $n\ge0$, where $K_{\ac,(n)}$ is the unique sub-extension of $K_\ac/K$ of degree $p^n$.
\item $\ord_{Y}\left({\Char_{\Lambda_{\cyc}}}\left(\Sel(\EC/K_\cyc)^\vee\right)\right)=1.$
\end{itemize}
\item if $p$ is a prime of good \emph{supersingular} reduction, suppose that
\begin{itemize}
\item $p$ splits in $K$ and both primes above $p$ in $K$ are totally ramified in $K_{\ac}/K$.
\item $a_p(\EC/\Q)=0$.
\item $\Sel(\EC/K)$is finite.
\end{itemize}
\end{itemize}
Then \ref{Mazur-conj} holds.
\end{lthm}

\begin{remark}[Comparison with results in \cite{KMS}]
Suppose that $\EC/\Q$ is \textit{ordinary} at $p$ and that when $\ord_{Y}\left(\Char_{\Lambda_\cyc}(\Sel(\EC/K_\cyc)^\vee)\right)=0$, the Iwasawa main conjecture predicts that the Mordell--Weil rank of
$\EC(K)$ is zero.
In this case, \ref{Mazur-conj} has already been established in \cite[Theorem~9.3]{KMS}.

When the Mordell--Weil rank of $\EC(K)$ is one, still assuming $\EC$ is \textit{ordinary} at $p$, Theorem~9.4 of \textit{op. cit.} proves that \ref{Mazur-conj} holds under certain technical hypotheses, which are circumvented in Theorem~\ref{Thm A} (see Remark~\ref{rk:compare} for a more detailed discussion).

Let $\lambda_\cyc$ be the $\lambda$-invariant of $\Sel(\EC/K_\cyc)^\vee$.
When $\lambda_\cyc=1$ and the sign of the functional equation of $L(\EC/K,s)$ at $s=1$ is $-1$, \cite[Theorem~1.5]{KMS} implies the validity of \ref{Mazur-conj}.
Since $\lambda_\cyc\ge\ord_Y$ in general, Theorem~\ref{Thm A} can cover cases different from those covered by \textit{op. cit.}
However, we need to assume a semi-cyclicity condition on $\Sel(\EC/K_\infty)^\vee$.
This condition of semi-cyclicity is related to a question asked by H.~Hida in \cite[\S~7.7.8]{hida2022_book}.
We provide sufficient conditions in \S~\ref{sec:SC hyp} and explicit examples where this condition of semi-cyclicity is satisfied.
In fact, our explicit examples give rise to Selmer groups whose Pontryagin duals are cyclic (not just semi-cyclic).
This provides a partial positive answer to Hida's question
where he asks for a series of systematic examples of Iwasawa modules which are expected to be cyclic for most primes.

Finally, \cite{KMS} tackles \ref{Mazur-conj} when $p$ is an \textit{ordinary} prime, while we handle both the \textit{ordinary} and \textit{supersingular} settings.
\end{remark}

The following is a special case of Theorem~\ref{Thm A}, which follows from combining the aforementioned theorem with results of A.~Matar--J.~Nekov\'{a}\v{r} \cite{matarnekovar}.

\begin{lthm}
\label{thm B}
Let $\EC/\Q$ be an elliptic curve of conductor $N_{\EC}$ and $K$ be an imaginary quadratic field such that the Heegner hypothesis is satisfied for $(K,N_\EC)$.
Let $p$ be an odd prime of good \emph{ordinary} reduction of $\EC$.
Further suppose that $p$ does not divide the Tamagawa number of $\EC/\Q$ and the basic Heegner point $y_K\notin \EC(K)_{\tors}$.
In addition, assume that
\begin{itemize}
\item $a_v(\EC/K) \not\equiv 1 \pmod{p}$ for every $v \mid p$ in $K$.
\item $\Sha(\EC/K)[p^\infty]=0$.
\item $\ord_{Y}\left({\Char_{\Lambda_{\cyc}}}(\Sel(\EC/K_\cyc)^\vee)\right)=1$.
\item $\Sha(\EC/K_{\ac,(n)})[p^\infty]$ is finite for all $n\ge0$, where $K_{\ac,(n)}$ is the unique sub-extension of $K_\ac/K$ of degree $p^n$.
\end{itemize}
Then \ref{Mazur-conj} holds.
\end{lthm}

We remind the reader that it follows from the work of B.~Gross--D.~Zagier and V.~Kolyvagin that the Mordell--Weil rank of $\EC(K)$ is 1 when the basic Heegner point $y_K\notin \EC(K)_{\tors}$.

\emph{Outlook}: Our elementary approach towards proving a long-standing conjecture of Mazur leads to sufficient conditions, which are easily verified through explicit computation.
Proving \ref{Mazur-conj} in full generality appears to be out of reach.
Thus, it might be interesting to investigate for what proportion of elliptic curves are the sufficient conditions from Theorem~\ref{Thm A} satisfied when $(K,p)$ are fixed.
In the same vein, if $(p,\EC)$ (resp. $(\EC,K)$) are fixed, can we calculate for what proportion of $K$ (resp. $p$) do the hypotheses of Theorem~\ref{Thm A} hold?

Appendix~\ref{appendix} is devoted to carrying out detailed calculations pertaining to properties of the signed Selmer groups over the anti-cyclotomic $\Zp$-extension of an imaginary quadratic field $K$.
Even though we can prove the supersingular analogues of most results required for proving Theorem~\ref{thm B}, our approach for proving \ref{Mazur-conj} falls short because when the Mordell--Weil rank of $\EC(K)$ is 1, we cannot show at present that the mixed signed Selmer groups are torsion over \emph{every} $\Zp$-extension of $K$.
Our approach will need to be refined further to handle this case.

Let $(\EC, K, p)$ be a tuple for which \ref{Mazur-conj} is valid.
Let $\cK$ be a $\Zp$-extension of $K$ where the Mordell--Weil rank of $\EC$ remains bounded.
Then there exists an integer $n_0(\cK)$ such that the rank of $\EC$ stops growing beyond the $n_0(\cK)$-th layer of this $\Zp$-extension.
When $\cK=K_{\cyc}$, it is known that $n_0(\cK)$ is independent of $p$, see \cite{chi02}.
It would be interesting to further investigate the variation of $n_0(\cK)$ as $\cK$ varies over all (but at most one) $\Zp$-extensions of $K$.

\section*{Acknowledgements}
The authors thank Francesc Castella, Haruzo Hida, S\"oren Kleine, Meng Fai Lim, Ahmed Matar, Barry Mazur, Robert Pollack, Jishnu Ray, and Stefano Vigni for helpful discussions during the preparation of this article.
We are grateful to S\"oren Kleine for sharing the preprint \cite{KMS} with us. We are also thankful to the referee for their helpful comments and suggestions on an earlier version of the article.
DK acknowledges the support of the AMS-Simons early career travel grant.
AL's research is supported by the NSERC Discovery Grants Program RGPIN-2020-04259 and RGPAS-2020-00096.

\section{Notation and Preliminaries}

\subsection{Elliptic curves and Selmer groups}
We collect notation and assumptions that will be in place throughout the article.

Let $\EC/\Q$ be an elliptic curve of conductor $N=N_{\EC}$ with good reduction at an odd prime $p$.
Let $K$ be an imaginary quadratic field.
The $\Zp^2$-extension of $K$ is denoted by $K_\infty$ and we write $G_\infty=\Gal(K_\infty/K)$.
The Iwasawa algebra $\Zp\llbracket G_\infty\rrbracket$ will be denoted by $\Lambda$.

Let $\Sigma$ be the set of primes in $\Q$ containing $p$ and all primes of bad reduction for $\EC$.
For any field $F/\Q$, define $\Sigma(F)$ to be the set of places of $F$ lying above those in $\Sigma$, and write $G_\Sigma(F)$ for the Galois group of the maximal extension of $F$ that is unramified outside of $\Sigma(F)$.
Furthermore, for any $v \in \Sigma$ and any finite extension $F/\Q$, write
\[
J_v(\EC/F)= \bigoplus_{w\mid v} H^1(F_w,\EC)[p^\infty].
\]
When $\mathcal{F}/F$ is an infinite extension of $F$, we set
\[
J_v(\EC/\mathcal{F}) = \varinjlim_{F\subseteq F' \subseteq \mathcal{F}} J_v(\EC/F').
\]

\begin{definition}\label{def:selmer-ord} Let $\EC/\Q$ be an elliptic curve.
Let $\Sigma$ be any finite set of primes containing those dividing $pN$.
For any extension $L/\Q$, define the Selmer group
\[
\Sel(\EC/L):=\ker \left( H^1(G_{\Sigma}(L),\EC[p^\infty]) \rightarrow
\prod_{v \in \Sigma} J_v(\EC/L) \right)
\]
\end{definition}

We have the following useful fact about Selmer groups and base change.

\begin{proposition}\label{prop:selmer-decomposition-ord}
Let $\EC/\Q$ be an elliptic curve, and let $\EC^{(K)}$ denote the twist of $E$ corresponding to $K$.
Then,
\[
\Sel(\EC/K_\cyc) \simeq \Sel(\EC/\Q_\cyc) \oplus \Sel(\EC^{(K)}/\Q_\cyc).
\]
\end{proposition}
\begin{proof}
It is a standard fact that
\[
H^1(K_\cyc,\EC[p^\infty]) \simeq H^1(\Q_\cyc, \EC[p^\infty]) \oplus H^1(\Q_\cyc,\EC^{(K)}[p^\infty]).
\]
A similar decomposition holds for local cohomology groups.
\end{proof}

\subsection{Preliminaries on Iwasawa modules}
\begin{definition}
Let $M$ be a finitely generated $\Lambda$-module.
Let $\fP$ be a prime ideal of $\Lambda$.
\item[\textup{i)}] The localization of $M$ at $\fP$ is denoted by $M_\fP$.
\item[\textup{ii)}] A $\Lambda$-module $M$ is called \textbf{pseudo-null} if $M_\fP=0$ for all prime ideals of $\Lambda$ of height $\le 1$.
Equivalently, if $\fP$ is a prime ideal such that $\ann_\Lambda(M)\subseteq \fP$, then the height of $\fP$ is at least 2.
\item[\textup{iii)}] The maximal pseudo-null submodule of $M$ is denoted by $M_{o}$.
\end{definition}

\begin{theorem}
\label{thm:structure}
Let $M$ be a direct sum of cyclic torsion $\Lambda$-modules such that $M_o=0$.
Then there exist principal ideals $I_1,\dots, I_m$ of $\Lambda$ such that $M\simeq \bigoplus_{i=1}^m\Lambda/I_i$.
\end{theorem}
\begin{proof}
 It suffices to show that the theorem holds when $M$ is a cyclic torsion $\Lambda$-module such that $M_o=0$.
 In this case, $M$ is of the form $\Lambda/I$ for some ideal $I$ of $\Lambda$.
 It remains to show that $I$ is principal.
 Suppose the contrary.
 Let us write $I=(f_1,\ldots ,f_r)$, where $f_i\in\Lambda$ and $r\ge2$.
 Since $\Lambda$ is a unique factorization domain, we can find a greatest common divisor of $f_1,\ldots, f_r$, which we denote by $h$.
 In particular, if we define $g_i=f_i/h$, $i=1,\ldots,r$, the elements $g_1,\ldots,g_r$ are coprime and the ideal $(g_1,\dots, g_r)$ is non-zero.

 Let $N$ be the submodule of $M$ given by $$N:=\Lambda h/I\subset \Lambda/I.$$
 Then $N\simeq\Lambda/(g_1,\ldots, g_r)$ as $\Lambda$-modules.
 We claim that $N$ is pseudo-null, which would then contradict the hypothesis that $M_o=0$.
 Suppose that $N$ is not pseudo-null.
 Then there exists a height-one prime $\fP$ such that $\ann_\Lambda(N)=(g_1,\ldots, g_r)\subseteq \fP$.
 Suppose that $\fP$ is generated by the element $f_0$, then $f_0 \mid g_i$ for all $i$.
 This contradicts the fact that the elements $g_1,\ldots,g_r$ are coprime.
 Hence, the claim on $I$ being principal follows.
\end{proof}

We shall work extensively with $\Lambda$-modules satisfying the semi-cyclicity hypothesis of Theorem~\ref{thm:structure} in the remainder of the article.
We shall say that such modules satisfy \eqref{hyp:SC}:
\vspace{0.2cm}
\begin{itemize}
 \item[(\mylabel{hyp:SC}{\textbf{S-C}})] $M$ is a direct sum of cyclic torsion $\Lambda$-modules.
\end{itemize}

\begin{definition}
Suppose that $M$ satisfies \eqref{hyp:SC} and $M_o=0$.
The \textbf{characteristic ideal} of $M$ is defined to be
\[
\Char_\Lambda(M)=\prod_{i=1}^m I_i,
\]
where $I_i$'s are the principal ideals given by Theorem~\ref{thm:structure}.
\end{definition}

\begin{remark}
We can define the characteristic ideal of any finitely generated torsion $\Lambda$-module, even if it does not satisfy \eqref{hyp:SC} and $M_o=0$; see \cite[Chapitre VII, no.4, Théorème~5]{bourbaki}.
\end{remark}

\subsection{\texorpdfstring{$\Zp$}{}-extensions}\label{sec:zp-extensions}
Let $K_\cyc$ and $K_\ac$ denote the cyclotomic and anti-cyclotomic $\Zp$-extension of $K$, respectively.
We fix two topological generators $\sigma$ and $\tau$ of $G_\infty$ such that
\begin{align*}
\overline{\langle \sigma\rangle}&=\ker\left(G_\infty\longrightarrow\Gal(K_\cyc/K)\right),\\
\overline{\langle \tau\rangle}&=\ker\left(G_\infty\longrightarrow\Gal(K_\ac/K)\right).
\end{align*}
Write $X+1=\sigma$ and $Y+1=\tau$.
Then we may identify $\Lambda$ with the ring of power series $\Zp\llbracket X,Y\rrbracket$.

\begin{lemma}
\label{lem:Zp-ext}
Let $\mathcal{K}/K$ be a $\Zp$-extension.
There exists a unique element $(a:b)\in\PP^1(\Zp)$ such that
\[
\overline{\langle \sigma^a\tau^b\rangle}=\ker\left(G_\infty\rightarrow\Gal(\mathcal{K}/K)\right),
\]
\end{lemma}

\begin{proof}
Each $\Zp$-extension corresponds to a unique subgroup of $H$ of $G_\infty$ such that $G_\infty/H\simeq\Zp$.
Since $G_\infty=\overline{\langle\sigma,\tau\rangle}\simeq\Zp^2$, the lemma follows.
\end{proof}

\begin{notn}
Let $\mathcal{K}/K$ be a $\Zp$-extension.
We adopt the following notation throughout the article.
\begin{itemize}
\item Set $\mathcal{K}=K_{a,b}$, where $(a:b)\in \PP^1(\Zp)$ is given as in Lemma~\ref{lem:Zp-ext}.
\item Set $\Gamma_{a,b}=\Gal(K_{a,b}/K)$, $H_{a,b}=\Gal(K_\infty/K_{a,b})$, and $\Lambda_{a,b}=\Zp\llbracket\Gamma_{a,b}\rrbracket$.
\item Set $\pi_{a,b}:\Lambda\rightarrow \Lambda_{a,b}$ for the map induced by the natural projection map $G_\infty\rightarrow\Gamma_{a,b}$.
\item Set $f_{a,b}=(1+X)^a(1+Y)^b-1$.
\item For $(a:b)=(1:0)$, set $\Gamma_\cyc$, $H_\cyc$, and $\pi_\cyc$ instead of $\Gamma_{(1:0)}$, $H_{(1:0)}$, and $\pi_{(1:0)}$, respectively.
\end{itemize}
\end{notn}
Note that $f_{a,b}$ is an irreducible element of $\Lambda$. Furthermore, if $(a:b)\neq (a':b')$, then $f_{a,b}$ and $f_{a',b'}$ are coprime to each other.

Let $\Lambda_\cyc=\Zp\llbracket \Gamma_\cyc\rrbracket \simeq \Zp\llbracket Y\rrbracket$ and $\Lambda_\ac=\Zp\llbracket \Gamma_\ac\rrbracket \simeq \Zp\llbracket X\rrbracket$.
If $M$ is a finitely generated $\Lambda_\cyc$-module, we write $\Char_{\Lambda_\cyc}(M)$ for its characteristic ideal and write $\ord_Y(M)$ for the multiplicity of $Y$ appearing in $\Char_{\Lambda_\cyc}(M)$.

\section{Mazur's growth number and the cyclotomic Selmer group - the ordinary case}
\label{sec: Cyclotomic - ord case}

The goal of this section is to prove the first half of Theorem~\ref{Thm A}.
Concretely, we prove sufficient conditions for \ref{Mazur-conj} when $\EC$ has good \textit{ordinary} reduction at $p$.

\begin{remark}
\label{clarifying Mazur's conj}
We clarify that in \cite{Maz84}, Mazur predicts that the growth number is 0 when either the $\Zp$-extension is \textit{not} the anti-cyclotomic $\Zp$-extension \textit{or} when the \textit{sign} of $(\EC,\chi_0)$ is +1; here $\chi_0$ is the principal character over $K$.
The notion of `sign' is defined in pp.~191--192 of \emph{op. cit.} and can be different from the sign of the functional equation of $L(\EC/K,s)$ precisely in the `exceptional case'.
What this means is that asking the sign of the functional equation of $L(\EC/K,s)$ to be +1 does not guarantee bounded growth in the exceptional case; but one also needs that the root number of $\EC/\Q$ to be +1.
Note that we write a simpler version of \ref{Mazur-conj} in the introduction because our result assumes \eqref{hyp: GHH}, i.e., we are in the `generic case'.
\end{remark}

Throughout, we impose the following non-anomalous condition:
\vspace{0.2cm}
\begin{itemize}
\item[(\mylabel{hyp:non-a}{\textbf{N-a}})] When $\EC$ has good \textit{ordinary} reduction at $p$, suppose that $p\nmid \#\widetilde{\EC}(\F_v)$.
Equivalently, for every $v\mid p$ in $K$ suppose that $a_v(\EC/K)\not\equiv 1\pmod{p}$.
\end{itemize}
\vspace{0.2cm}

We remark that this is a mild hypothesis to impose; see \cite[Section~9.1]{KLR}.
Note also that this non-anomalous condition ensures that $\EC(K)[p]=0$.

\begin{proposition}[Control Theorem - the ordinary case]
\label{prop:control-ord}
Let $\EC/\Q$ be an elliptic curve and $p$ be an odd prime of good \emph{ordinary} reduction such that \eqref{hyp:non-a} holds.
Let $\cK=K_{a,b}$ for some $(a \colon b) \in \mathbb{P}^1(\Zp)$ and let $H=H_{a,b}$.
The natural restriction map then induces an isomorphism
\[
\Sel(\EC/\cK)\simeq \Sel(\EC/K_\infty)^H.
\]
\end{proposition}
\begin{proof}

Consider the commutative diagram
\[
\begin{tikzcd}
0 \arrow[r] &
\Sel(\EC/\cK) \arrow[r] \arrow[d,"\alpha"] &
H^1(G_{\Sigma}(\cK),\EC[p^\infty]) \arrow[r,"\pi"] \arrow[d,"\beta"] &
\prod_{v \in \Sigma(\cK)} J_v(\EC/\cK) \arrow[d,"\gamma=
\prod \gamma_v"]
\\
0 \arrow[r] &
\Sel(\EC/K_\infty)^{H} \arrow[r] &
H^1(G_{\Sigma}(K_\infty),\EC[p^\infty])^{H} \arrow[r] &
\prod_{v \in \Sigma(K_\infty)} J_v(\EC/K_\infty)^{H}
\end{tikzcd}
\]

The snake lemma gives the exact sequence
\[
0 \longrightarrow \ker(\alpha) \longrightarrow \ker(\beta) \longrightarrow \ker(\gamma) \cap \image(\pi) \longrightarrow \coker(\alpha) \longrightarrow \coker(\beta),
\]
so the result will follow upon showing that $\ker(\beta)$, $\ker(\gamma)\cap \image(\pi)$, and $\coker(\beta)$ are zero.

First, we note that by the inflation-restriction exact sequence, $\coker(\beta)\hookrightarrow H^2(H, \EC(K_\infty)[p^\infty])$.
Since $H$ has $p$-cohomological dimension 1, the second cohomology group is trivial and it follows that $\coker(\beta)=0$.
The inflation-restriction exact sequence also implies
\[
\mathrm{ker}(\beta ) = H^1(H, \EC(K_\infty)[p^\infty]).
\]
Since $v\mid p$ is
deeply ramified in the sense of Coates--Greenberg, imitating the proof of \cite[Proposition~4.8]{CG96}
we obtain that
\[
\ker(\gamma) = \prod \ker(\gamma_v) = \prod_{\substack{v \in \Sigma(\cK)\\ v \nmid p}} H^1(H_{w},\EC(K_{\infty,w})[p^\infty]) \times \prod_{\substack{v \in \Sigma(\cK)\\ v\mid p}} H^1(H_{w},\widetilde{\EC}(k_{w})[p^\infty]),
\]
where $w=w_v$ denotes any place of $K_\infty$ lying above $v$, $H_{w}$ denotes the decomposition group of $w$ inside $H$, $\widetilde\EC$ denotes the reduction of $\EC$ modulo $w$, and $k_{w}$ denotes the residue field of $K_\infty$ at $w$.

Since we have assumed \eqref{hyp:non-a}, for every $v \mid p$,
\[
\widetilde{\EC}(\F_v)[p]=0.
\]
Then by \cite[Lemma 5.11]{mazur72} we have
\[
\EC(K_v)[p]=0.
\]
The argument given in \cite[Lemma 5.2]{dionray} now implies
 \[
\EC(K_{\infty,w})[p]=0
\]
for $w\mid v$.
Applying \cite[Lemma 5.11]{mazur72} once again, we have
\[
\widetilde{\EC}(K_{w})[p]=0.
\]
Thus, for $v \mid p$, we have $\ker(\gamma_v)=0$.

Since $\EC(K_\infty)[p^\infty]$ is contained in $\EC(K_{\infty,w})[p^\infty]$ for any place $w \in \Sigma(\cK)$, and
in particular for $w\mid p$, the above argument also shows that $\ker(\beta)=0$.

It only remains to show that $\gamma_v$ is injective for $v \nmid p$.
In fact, for $v \nmid p$, the local map $\gamma_v$ is simply the identity map: for any $v \in \Sigma(\cK)$ such that $v\nmid p$ and any place $w$ of $K_\infty$ lying above $v$, we have that $K_{\infty,w}=\cK_{v}$ is the unique $\Z_p$-extension of $K_{v_0}$, where $v_0$ is the place of $K$ lying below $v$.
\end{proof}

\begin{proposition}
\label{prop:ordinary-fg-m0}
Let $\EC/\Q$ be an elliptic curve and $p$ be an odd prime of good \emph{ordinary} reduction satisfying \eqref{hyp:non-a}.
Then $\Sel(\EC/K_\infty)^\vee$ (resp. $\Sel(\EC/K_\cyc)^\vee$) is a finitely generated torsion $\Lambda$-module (resp. $\Lambda_\cyc$-module).
Furthermore, if $\Sel(\EC/K_\infty)^\vee$ satisfies \eqref{hyp:SC}, then
\[
\pi_\cyc(\Char_\Lambda(\Sel(\EC/K_\infty)^\vee))=\Char_{\Lambda_\cyc}(\Sel(\EC/K_\cyc)^\vee).
\]
\end{proposition}

\begin{proof}
Let $M=\Sel(\EC/K_\infty)^\vee$ and $M_\cyc=\Sel(\EC/K_\cyc)^\vee$.
Since $\EC$ has good ordinary reduction at $p$, the fact that $M_\cyc$ is a finitely generated torsion $\Lambda_\cyc$-module is due to Kato (see \cite[Theorem~17.4]{kato04}) since by Proposition \ref{prop:selmer-decomposition-ord} we have
\begin{equation}
\label{eq:selmer-decomp}
\Sel(\EC/K_\cyc)\simeq \Sel(\EC/\Q_\cyc)\oplus \Sel(\EC^{(K)}/\Q_\cyc).
\end{equation}
As explained in \cite[Remark~2.2]{KLR}, it now follows from the work of Balister-Howson \cite{BH97} (see also \cite{HO10}) that $\Sel(\EC/K_\infty)^\vee$ is a finitely generated torsion $\Lambda$-module.

Our running hypothesis \eqref{hyp:non-a}, together with \cite[Proposition~6.1]{KLR}, imply that $M_o=0$.
The last assertion of the proposition now follows from Theorem~\ref{thm:structure}.
\end{proof}

\begin{remark}
The proof of the last assertion of Proposition~\ref{prop:ordinary-fg-m0} relies crucially on Hypothesis~\eqref{hyp:SC}.
Let $M=\Sel(\EC/K_\infty)^\vee$ and suppose that the aforementioned hypothesis does not hold.
By \cite[Chapitre VII, no.4, Théorème 5]{bourbaki}, we know that there exist principal ideals $I_1,\dots, I_m$ of $\Lambda$ and a pseudo-null $\Lambda$-module $N$ sitting inside a short exact sequence
\[
0\longrightarrow M\longrightarrow\bigoplus_{i=1}^m\Lambda/I_i\longrightarrow N\longrightarrow 0.
\]
This induces the exact sequence of $\Lambda_\cyc$-modules
\[
H_1(H_\cyc,N)\longrightarrow M_{H_\cyc}\longrightarrow\bigoplus_{i=1}^m\Lambda_\cyc/\pi_\cyc(I_i)\longrightarrow H_0(H_\cyc,N)\longrightarrow 0.
\]
In the case where the $\Lambda_\cyc$-characteristic ideals of $H_i(H_\cyc,N)$ are non-trivial, it is possible that $\Char_{\Lambda_\cyc}(\Sel(\EC/K_\cyc)^\vee)=\Char_{\Lambda_\cyc}(M_{H_\cyc})$ is not equal to $\pi_\cyc\left(\Char_{\Lambda}(\Sel(\EC/K_\infty)^\vee)\right)=\pi_\cyc(\prod_{i=1}^mI_i)$.
\end{remark}

\begin{lemma}
\label{lem:rank-char}
Suppose that $\Sel(\EC/K_\infty)^\vee$ satisfies \eqref{hyp:SC} and that \eqref{hyp:non-a} holds.
Let $\mathcal{K}=K_{a,b}$ be a $\Zp$-extension of $K$ with unbounded Mordell--Weil ranks of $\EC$.
Then $(f_{a,b})\mid \Char_\Lambda\left(\Sel(\EC/K_\infty)^\vee\right)$ as ideals of $\Lambda$.
\end{lemma}
\begin{proof}
Since we are assuming \eqref{hyp:SC}, Theorem~\ref{thm:structure} tells us that
\[
\Sel(\EC/K_\infty)^\vee\simeq \bigoplus_{i=1}^m \Lambda/(g_i)
\]
for some $g_i \in \Lambda$. By Proposition~\ref{prop:control-ord},
\[
\Sel(\EC/K_\infty)^\vee_{H_{a,b}}\simeq \Sel(\EC/K_{a,b})^\vee\simeq \bigoplus_{i=1}^m \Lambda/(g_i,f_{a,b})\simeq \bigoplus_{i=1}^m \Lambda_{a,b}/(\pi_{a,b}(g_i)).
\]
Therefore, $\Sel(\EC/K_{a,b})^\vee$ is $\Lambda_{a,b}$-torsion if and only if $f_{a,b}\nmid \prod_{i=1}^mg_i$, which is a generator of the characteristic ideal $\Char_\Lambda\left(\Sel(\EC/K_\infty)^\vee\right)$.

If the Mordell--Weil ranks of $\EC$ are unbounded in $K_{a,b}$, then $(\EC(K_{a,b})\otimes\Qp/\Zp)^\vee$ is not $\Lambda_{a,b}$-torsion.
In particular, neither is $\Sel(\EC/K_{a.b})^\vee$, which proves the lemma.
\end{proof}

\begin{remark} In \cite[Lemma~2.2]{KMS}, a similar result is proven without assuming \eqref{hyp:SC}.
\end{remark}

\begin{definition}\label{def:growth-number}
Write $\sM(\EC,K)$ for the set of $(a:b)\in\PP^1(\Zp)$ such that the Mordell--Weil ranks of $\EC$ are unbounded in $K_{a,b}$.
\textbf{Mazur's growth number} is the cardinality of the set $\sM(\EC,K)$ and is denoted by $n(\EC,K)$.
\end{definition}

\begin{remark}
Since $\Sel(\EC/K_\infty)^\vee$ is a finitely generated torsion $\Lambda$-module, it follows from Lemma~\ref{lem:rank-char} that $\sM(\EC,K)$ is finite under \eqref{hyp:SC} (see also \cite[Theorem~1.5]{KMS}).
\end{remark}

\begin{proposition}
\label{prop:ord-count}
With notation introduced above,
\[
n(\EC,K)\le \ord_{Y}\left(\Char_{\Lambda_\cyc}\left(\Sel(\EC/K_\cyc)^\vee\right)\right).
\]
\end{proposition}

\begin{proof}
If $(a:b)\ne (a':b')$ are elements of $\PP^1(\Zp)$, then $f_{a,b}$ and $f_{a',b'}$ are coprime elements of $\Lambda$.
It follows from Lemma~\ref{lem:rank-char} that
\begin{equation}
\label{eq:product_divides}
\left.\prod_{(a:b)\in \sM(\EC,K)}f_{a,b}\right\vert\Char_\Lambda\left(\Sel(\EC/K_\infty)^\vee\right).
\end{equation}
Since $(1:0)\notin\sM(\EC,K)$, it follows that
\[
Y \mid \pi_{\cyc}(f_{a,b})=(1+Y)^b-1.
\]
in $\Lambda_\cyc$ for all $(a:b)\in\sM(\EC,K)$.
Hence,
\[
Y^{n(\EC,K)}\left|\prod_{(a:b)\in \sM(\EC,K)}\pi_\cyc(f_{a,b}).\right.
\]
By \eqref{eq:product_divides}, the product on the right-hand side divides
\[\pi_\cyc\left(\Char_\Lambda\left(\Sel(\EC/K_\infty)^\vee\right)\right)=\Char_{\Lambda_\cyc}\left(\Sel(\EC/K_\cyc)^\vee\right),
\]
where the last equality is given by the last assertion of Proposition~\ref{prop:ordinary-fg-m0}.
This concludes the proof.
\end{proof}

Before stating our corollary on \ref{Mazur-conj}, we recall the statement of the modified Heegner hypothesis:
\vspace{0.2cm}
\begin{itemize}
\item[(\mylabel{hyp: GHH}{\textbf{GHH}})] When the sign of the functional equation of $L(\EC/K,s)$ is $-1$, the modified Heegner hypothesis holds with respect to the pair $(K,N_{\EC})$, that is, if $N=N^+ N^-$ such that $N^+$ is the largest factor of $N$ divisible only by primes that are split in $K$, then $N^-$ is the square-free product of an \emph{even} number of primes all of which are inert in $K$.
\end{itemize}
\vspace{0.2cm}

\begin{corollary}
\label{cor:main-result-ord}
Let $\EC/\Q$ be an elliptic curve and $p$ be an odd prime of good \emph{ordinary} reduction such that  \eqref{hyp: GHH} and \eqref{hyp:non-a} hold.
Furthermore, suppose that $\Sel(\EC/K_\infty)^\vee$ satisfies \eqref{hyp:SC}.
If either one of the following two conditions holds
\begin{itemize}
\item $\ord_{Y}\left(\Char_{\Lambda_\cyc}\left(\Sel(\EC/K_\cyc)^\vee\right)\right)=0$,
\item the sign of the functional equation of $L(\EC/K,s)$ is $-1$, $\ord_{Y}\left(\Char_{\Lambda_\cyc}\left(\Sel(\EC/K_\cyc)^\vee\right)\right)=1$, and $\Sha(\EC/K_{\ac,(n)})[p^\infty]$ is finite for all $n\ge0$, where $K_{\ac,(n)}$ is the unique sub-extension of $K_\ac/K$ of degree $p^n$,
\end{itemize}
then \ref{Mazur-conj} holds.
\end{corollary}

\begin{proof}
If $\ord_{Y}\left(\Char_{\Lambda_\cyc}\Sel(\EC/K_\cyc)^\vee\right)=0$, then Proposition~\ref{prop:ord-count} implies that $n(\EC,K)=0$.
In particular, the Mordell--Weil ranks of $\EC$ are bounded in all $\Zp$-extensions of $K$.
In this case, \ref{Mazur-conj} holds trivially.

If the sign of the functional equation of $L(\EC/K,s)$ is $-1$, Hypothesis \eqref{hyp: GHH} implies that $\Sel(\EC/K_\ac)^\vee$ has rank one over $\Lambda_\ac$.
As explained in \cite[Appendix~A.1]{LongoVigni2}, the works of J.~Nekov\'{a}\u{r} \cite{nekovar} and V.~Vatsal \cite{vatsal} imply that $\rank_{\Zp}\Sel(\EC/K_{\ac,(n)})^\vee=p^n+O(1)$ (this neither depends on the reduction type of $\EC$ at $p$ nor does it require $p>5$); see also \cite[Theorem~A]{Bertolini-Compositio}.

The hypothesis on $\Sha(\EC/K_{\ac,(n)})$ implies that the Mordell--Weil ranks of $\EC$ are unbounded in $K_\ac$.
If in addition $\ord_{Y}\left(\Char_{\Lambda_{\cyc}}\left(\Sel(\EC/K_\cyc)^\vee\right)\right)=1$, then Proposition~\ref{prop:ord-count} asserts that $n(\EC,K)=1$.
Therefore, \ref{Mazur-conj} holds.
\end{proof}

\begin{remark}\label{rk:compare}
The case where $\ord_{Y}\left(\Char_{\Lambda_\cyc}\left(\Sel(\EC/K_\cyc)^\vee\right)\right)=0$ has already been covered by \cite[Theorem~9.3]{KMS}.
In \textit{loc. cit.}, it is assumed that $\Sel(\EC/K)$ is finite, which is equivalent to our condition on the order of vanishing of the characteristic ideal at $Y=0$ (by Mazur's control theorem from \cite{mazur72}).
Under this hypothesis, $\Sel(\EC/K_{a,b})^\vee$ is torsion over $\Lambda_{a,b}$ for all $(a:b)\in\PP^1(\Zp)$.
In particular, the Mordell--Weil ranks of $\EC$ are bounded over $K_{a,b}$.
In particular, \ref{Mazur-conj} holds.
Note that it is not necessary to assume that $\Sel(\EC/K_\infty)^\vee$ satisfies \eqref{hyp:SC} in this case.

In the second case covered by Corollary~\ref{cor:main-result-ord}, the Mordell--Weil of $\EC$ over $K$ is one.
Our result is complementary to \cite[Theorem~9.4]{KMS}, where  \ref{Mazur-conj} has been proved under a different set of hypotheses (namely $\rank_{\Z}\EC(K)=1$, the prime $p$ does not divide the Tamagawa numbers of $\EC$ and the quadratic twist $\EC^{(K)}$, $\Sha(\EC/K)[p^\infty]=0$,  \eqref{hyp:non-a}  and certain conditions on the $p$-adic regulators of $\EC$ and $\EC^{(K)}$ are assumed to hold).
\end{remark}

\section{Mazur's growth number and signed Selmer groups}
\label{sec: Cyclotomic - ss case}
The goal of this section is to prove the second half of Theorem~\ref{Thm A}.
More specifically, we prove sufficient conditions for \ref{Mazur-conj} under the hypothesis that $\EC$ has good \textit{supersingular} reduction at $p$ with $a_p(\EC/\Q)=0$.
Recall that in view of the Hasse bound, this condition on $a_p(\EC/\Q)$ is automatically satisfied when $p\geq 5$.

In this supersingular setting, we assume that $p$ splits in $K$ as $p=\fp \fq$, and we further assume that both these primes above $p$ are totally ramified in $K_\ac/K$.
This condition is satisfied for example, if we assume that $p$ does not divide the class number of $K$.
For $\fr\in\{\fp,\fq\}$, given any sub-extension $\cK$ of $K_\infty/K$, there is exactly one place of $\cK$ lying above $\fr$. We denote this place by the same symbol.

\subsection{Defining signed Selmer groups}
We first recall the definition of plus and minus subgroups given in \cite[\S2.3]{kim14}, which generalize the ones initially defined in \cite{kobayashi03}.

\begin{definition}
Let $m,n\ge0$ be integers.
Set $K_{(m,n)}$ to denote the intersection of $K_\infty$ and the ray class field of $K$ modulo $\fp^{m+1}\fq^{n+1}$.
For $\fr\in\{\fp,\fq\}$, write $K_{(m,n),\fr}$ for the completion of $K_{(m,n)}$ at the unique prime above $\fr$.

Given $m$ and $n$ as above, define
    \begin{align*}
        \EC^\pm(K_{(m,n),\fp})&=\left\{P\in\widehat{\EC}(K_{(m,n),\fp}):\Tr_{K_{(m,n),\fp}/K_{(\ell+1,n),\fp}}(P)\in\widehat{\EC}(K_{(\ell,n),\fp})\, \, \forall \ell\in S_m^\pm\right\},\\
        \EC^\pm(K_{(m,n),\fq})&=\left\{P\in\widehat{\EC}(K_{(m,n),\fq}):\Tr_{K_{(m,n),\fq}/K_{(m,\ell+1),\fq}}(P)\in\widehat{\EC}(K_{(m,\ell),\fq})\, \, \forall \ell\in S_n^\pm\right\},
    \end{align*}
    where
    \begin{align*}
        S_r^+&=\{\ell\in 2\Z:0\le \ell <r\},\\
        S_r^-&=\{\ell\in2\Z+1:1\le \ell<r\},
    \end{align*}
for $r\in\{m,n\}$.

For $\fr\in\{\fp,\fq\}$, let $H^1_\pm(K_{(m,n),\fr},\EC[p^\infty])$ be the image of $\EC^\pm(K_{(m,n),\fr})\otimes\Qp/\Zp$ in $H^1(K_{(m,n),\fr},\EC[p^\infty])$ under the Kummer map.
  \end{definition}
We define four signed Selmer groups:
\begin{definition}\label{def:2-var-Sel}
Given $\bullet,\star\in\{+,-\}$ and integers $m,n\ge0$, define the signed Selmer group
\[
\Sel^\bc(\EC/K_{(m,n)})=\ker\left(\Sel(\EC/K_{(m,n)})\longrightarrow\frac{H^1(K_{(m,n),\fp},\Ep)}{H^1_\bullet(K_{(m,n),\fp},\Ep)}\times \frac{H^1(K_{(m,n),\fq},\Ep)}{H^1_\star(K_{(m,n),\fq},\Ep)}\right).
\]
For $\cK\in\{K_{\infty}\}\cup\{K_{a,b} \mid (a:b)\in\PP^1(\Zp)\}$, define
\[\Sel^\bc(\EC/\cK)=\varinjlim_{K_{(m,n)}\subset \cK }\Sel^\bc(\EC/K_{(m,n)}),\]
where the connecting maps for the direct limit are given by restrictions.
\end{definition}

\begin{remark}\label{rk:same}
For all $\bullet,\star\in\{+,-\}$, we have $\Sel^\bc(\EC/K)=\Sel(\EC/K)$ since $\EC^\pm(K_\fr)=\widehat{E}(K_\fr)$ when $\fr\in\{\fp,\fq\}$.
\end{remark}

\subsection{Bounding Mordell--Weil ranks in terms of signed Selmer groups}
We discuss how we may bound the Mordell--Weil ranks of $\EC$ over $\Zp$-extensions of $K$ via the signed Selmer groups.
\begin{lemma}
\label{lem:bound-rank-pm}
Let $m,n\ge0$ be integers.
Then
\[
\rk_{\Z} \EC(K_{(m,n)})\le \sum_{\bullet,\star\in\{+,-\}}\rank_{\Zp}\Sel^\bc(\EC/K_{(m,n)})^\vee.
\]
\end{lemma}

\begin{proof}
This proof is similar to \cite[Theorem~4.3]{leisprung}, which relies on the short exact sequence
\[
0\longrightarrow \EC^+(K_{(m,n),\fr})\cap\EC^-(K_{(m,n),\fr})\longrightarrow
\EC^+(K_{(m,n),\fr})\oplus\EC^-(K_{(m,n),\fr})\longrightarrow \widehat\EC(K_{(m,n),\fr})\longrightarrow 0
\]
for $\fr\in\{\fp,\fq\}$, resulting in a group homomorphism
\[
\bigoplus_{\bullet,\star\in\{+,-\}}\Sel^\bc(\EC/K_{(m,n)})\longrightarrow\Sel(\EC/K_{(m,n)})
\]
whose cokernel is of finite cardinality.
\end{proof}

\begin{remark}
\label{rem: basically what antonio said}
We need the assumption that both primes above $p$ are totally ramified in $K_{\ac/K}$ to appeal to the results in \cite{leisprung}.
Although we believe that this assumption can potentially be relaxed, one would have to rewrite some of the proofs and define the Selmer groups taking into account the splitting of the primes above $p$ in $K_{(m,n)}$.
\end{remark}

\begin{corollary}
\label{cor:ss-bounded-torsion}
Let $(a:b)\in\PP^1(\Zp)$.
Suppose that $\Sel^\bc(\EC/K_{a,b})^\vee$ is $\Lambda_{a,b}$-torsion for all four choices of $\bullet,\star\in\{+,-\}$.
Then the Mordell--Weil ranks of $\EC$ are \emph{bounded} over sub-extensions of $K_{a,b}/K$.

\end{corollary}

\begin{proof}
Let $\cK=K_{a,b}$ and for any non-negative integer $n$, write $\cK_n$ for the sub-extension of $\cK/K$ of degree $p^n$.
Since $\EC(\cK)[p^\infty]=0$ by \cite[Lemma~3.4]{lei2021akashi}, it follows that the map $H^1(G_\Sigma(\cK_n),\EC[p^\infty])\rightarrow H^1(G_\Sigma(\cK),\EC[p^\infty])$ is injective by the inflation-restriction sequence.
This in turn induces an injection
\[
\Sel^\bc(\EC/\cK_n)\hookrightarrow  \Sel^\bc(\EC/\cK)
\]
for all four choices of $\bullet,\star\in\{+,-\}$.
In particular, $\rank_{\Zp}\Sel^\bc(\EC/\cK_n)^\vee$ is bounded independently of $n$ under the assumption that $\Sel^\bc(\EC/\cK)^\vee$ is $\Lambda_{a,b}$-torsion.
The 
corollary now follows from Lemma~\ref{lem:bound-rank-pm}.
\end{proof}

\subsection{Mazur's growth number}

We study Mazur's growth number under the assumption that $\Sel(\EC/K)$ is finite.

\begin{proposition}
\label{prop:torsion-finite}
Suppose that $\Sel(\EC/K)$ is finite.
Then $\Sel^\bc(\EC/K_{a,b})^\vee$ is torsion over $\Lambda_{a,b}$ for all $\bullet,\star\in\{+,-\}$ and $(a,b)\in\PP^1(\Zp)$.
\end{proposition}

\begin{proof}
It follows from Remark~\ref{rk:same} that we have the following commutative diagram
\[
\begin{tikzcd}
0 \arrow[r] &
\Sel(\EC/K) \arrow[r] \arrow[d,"\alpha"] &
H^1(G_{\Sigma}(K),\EC[p^\infty]) \arrow[r] \arrow[d,"\beta"] &
\prod_{v \in \Sigma(K)} J_v(\EC/\cK) \arrow[d,"\gamma=
\prod \gamma_v^\bc"]
\\
0 \arrow[r] &
\Sel^\bc(\EC/K_{a,b})^{\Gamma_{a,b}} \arrow[r] &
H^1(G_{\Sigma}(K_{a,b}),\EC[p^\infty])^{\Gamma_{a,b}} \arrow[r] &
\prod_{v \in \Sigma(K_\infty)} J^\bc_v(\EC/K_{a,b})^{\Gamma_{a,b}} ,
\end{tikzcd}
\]
where $J^\bc_v(\EC/K_{a,b})$ is defined by
\[
\begin{cases}
    J_v(\EC/K_{a,b})&v\nmid p,\\
    \displaystyle\varinjlim_{K_{(m,n)}\subset K_{a,b}}\frac{H^1(K_{(m,n),\fp},\Ep)}{H^1_\bullet(K_{(m,n),\fp},\Ep)}&v=\fp,\\
    \displaystyle\varinjlim_{K_{(m,n)}\subset K_{a,b}}\frac{H^1(K_{(m,n),\fq},\Ep)}{H^1_\star(K_{(m,n),\fq},\Ep)}&v=\fq.
\end{cases}
\]
We show that $\alpha$ is injective, with finite cokernel.

The map $\beta$ is an isomorphism as in the proof of Corollary~\ref{cor:ss-bounded-torsion}.
When $v\nmid p$, the map $\gamma^\bc_v$ has finite kernel by \cite[Lemma~3.3]{greenberg99}.
If $\fr\in\{\fp,\fq\}$, we have injections
\begin{align*}
H^1(K_\fr,\EC[p^\infty])\hookrightarrow
H^1(K_{(m,n),\fr},\EC[p^\infty])\\
H^1_\pm(K_\fr,\EC[p^\infty])\hookrightarrow
H^1_\pm(K_{(m,n),\fr},\EC[p^\infty])
\end{align*}
induced by the restriction map.
Hence, $\gamma_v^\bc$ is injective when $v \mid p$.
Therefore, our claim on $\alpha$ follows from the snake lemma.
In particular, $\Sel^\bc(\EC/K_{a,b})^{\Gamma_{a,b}}$ is finite if $\Sel(\EC/K)$ is finite.
This implies that $\Sel^\bc(\EC/K_{a,b})^\vee$ is torsion over $\Lambda_{a,b}$ as desired.
\end{proof}

\begin{corollary}\label{cor:main-result-ss}
Let $\EC/\Q$ be an elliptic curve and $p$ be an odd prime of good \emph{supersingular} reduction with $a_p(\EC/\Q)=0$ such that $\Sel(\EC/K)$ is finite.
Then \ref{Mazur-conj} holds.
\end{corollary}
\begin{proof}
By Proposition~\ref{prop:torsion-finite} we know that $\Sel^\bc(\EC/K_{a,b})^\vee$ is torsion over $\Lambda_{a,b}$ for all $\bullet,\star\in\{+,-\}$ and all $(a,b)\in\PP^1(\Zp)$ when $\Sel(\EC/K)$ is finite.
Therefore, by Corollary~\ref{cor:ss-bounded-torsion}, the Mordell--Weil ranks of $\EC$ are bounded over all $\Zp$-extensions of $K$.
In particular, \ref{Mazur-conj} holds.
\end{proof}

\section{Sufficient conditions for the hypothesis \texorpdfstring{\eqref{hyp:SC}}{(S-C)}}
\label{sec:SC hyp}

We record results on Selmer groups that will allow us to apply Corollary~\ref{cor:main-result-ord} to examples in $\S~\ref{sec:examples}$.
In particular, we only consider the good \emph{ordinary} case.
We refer the reader to Appendix~\ref{appendix} for analogous results in the supersingular case.

\begin{proposition}
\label{prop:SC-ord}
Let $\EC/\Q$ be an elliptic curve and let $p$ be an odd prime number where $\EC$ has good ordinary reduction.
Let $K$ be an imaginary quadratic field.
Suppose that $\lambda(\Sel(\EC/K_\cyc)^\vee)\le 1$ and $\mu(\Sel(\EC/K_\cyc)^\vee)=0$, then $\Sel(\EC/K_\infty)^\vee$ satisfies \eqref{hyp:SC}.
\end{proposition}

\begin{proof}
It follows from \cite[Proposition 4.8]{greenberg99}  that $\Sel(\EC/K_\cyc)$ has no proper $\Lambda_\cyc$-submodules of finite index.
Thus, if $\lambda(\Sel(\EC/K_\cyc)\le 1$ and $\mu(\Sel(\EC/K_\cyc))=0$, then $\Sel(\EC/K_\cyc) \simeq \Qp/\Zp$ or $0$.
In particular, it is co-cyclic over $\Lambda_\cyc$.
Therefore, it follows from Proposition~\ref{prop:control-ord} and Nakayama's lemma (see \cite{BH97}) that $\Sel(\EC/K_\infty)^\vee$ is cyclic over $\Lambda$.
\end{proof}

\begin{remark}
In view of \cite[Corollary~6.5]{KR2} we expect most of the elliptic curves of rank at most 1 over $K$ to satisfy \eqref{hyp:SC}.
To find the precise proportion, further calculations are required.
\end{remark}

\begin{theorem}
\label{thm: matar-nekovar 4.8}
Let $\EC/\Q$ be an elliptic curve of conductor $N_{\EC}$ and let $K/\Q$ be an imaginary quadratic field.
Let $p$ be an odd prime of good \emph{ordinary} reduction of $\EC$ satisfying \eqref{hyp:non-a} for which $p \nmid N_{\EC}$.
Further suppose that
\begin{itemize}
\item $p$ does not divide the Tamagawa number of $\EC/\Q$.
\item $\rank_{\Z} \EC(K)=1$ and $\Sha(\EC/K)[p^\infty]=0$.
\item the basic Heegner point $y_K\notin \EC(K)_{\tors}$.
\end{itemize}
Then, $\Sha(\EC/K_{\ac})[p^\infty]=0$ and $\Sel(\EC/K_{\ac})^\vee$ is a free module with $\Lambda_{\ac}$-rank equal to 1.
\end{theorem}

\begin{proof}
See \cite[Theorem~4.8]{matarnekovar}.
\end{proof}

\begin{remark}
When $\Sel(\EC/K_{\ac})^\vee$ is a free module with $\Lambda_{\ac}$-rank equal to 1, it is cyclic over $\Lambda_\ac$.
Then $\Sel(\EC/K_{\infty})^\vee$ satisfies \eqref{hyp:SC} (once again by Nakayama's lemma and the control theorem).
\end{remark}

\section{Explicit examples}\label{sec:examples}

In this section, we demonstrate how Corollaries \ref{cor:main-result-ord} and \ref{cor:main-result-ss} can be used to verify \ref{Mazur-conj} in concrete examples.
The Magma \cite{Magma} code which we used to check these examples is available upon request.

We shall identify the Iwasawa algebra $\Zp\llbracket\Gal(\Q_\cyc/\Q)\rrbracket$ with $\Lambda_\cyc$.
Furthermore, we fix a topological generator $\gamma$ of $\Gal(\Q_\cyc/\Q)$ so that $\Zp\llbracket\Gal(\Q_\cyc/\Q)\rrbracket$ may be written as $\Zp\llbracket T\rrbracket$, where $T=\gamma-1$.

\begin{remark}
\label{Lp-prec}
All coefficients of the $p$-adic $L$-functions appearing in this article are computed to Magma's internal precision of $O(p^{33})$ for the constant coefficients and $O(p^4)$ for all others.
In particular, our computations only yield \textit{lower bounds} on the orders of vanishing for $p$-adic $L$-functions.
However, we explain below how we can use these lower bounds to deduce equalities.

\end{remark}

\begin{remark}\label{remark-sha}
Recall that our results rely on the assumption that $\Sha(\EC/K_{\ac,(n)})[p^\infty]$ is finite for all $n\geq0$.
This is a widely-believed conjecture, but it is the one hypothesis which we cannot computationally verify.
For the examples that follow, we assume that this condition holds.
\end{remark}

\subsection{Ordinary reduction}
\label{sec:examples-ord}
Before presenting our examples, we briefly list the sufficient conditions needed to verify the hypotheses of Corollary \ref{cor:main-result-ord}.
If $\EC/\Q$ has good \textit{ordinary} reduction at $p$ and $K$ is a fixed imaginary quadratic field, write $\EC^{(K)}/\Q$ for the twist of $\EC$ by the quadratic character corresponding to $K$ and $L_p(\EC),L_p(\EC^{(K)})\in\Z_p\llbracket T\rrbracket$
for the $p$-adic $L$-functions attached to $\EC/\Q$ and $\EC^{(K)}/\Q$, respectively.
In addition to our running hypothesis on $\Sha$ (see Remark \ref{remark-sha}), it suffices to check the following conditions for the tuple $(\EC,K,p)$:
\begin{enumerate}
\setcounter{enumi}{-1}
\item $a_v(\EC/K)\not\equiv 1\pmod{p}$ for each $v \mid p$ in $K$.
\item $\EC/K$ has root number $-1$.
\item $(K,N_\EC)$ satisfies (\ref{hyp: GHH}).
\item $\lambda\big(L_p(\EC)\big)+\lambda\big(L_p(\EC^{(K)})\big)=1$ and $\mu\big(L_p(\EC)\big)=\mu\big(L_p(\EC^{(K)})\big)=0$.
\item $\ord_TL_p(\EC)+\ord_TL_p(\EC^{(K)})\geq 1$.
\end{enumerate}
By the Iwasawa main conjecture (a theorem in our setting by \cite{SkinnerUrban}), observe that condition~(3) together with the isomorphism
\begin{equation}
\label{eq:selmer-decomp-ord}
\Sel(\EC/K_\cyc) \simeq \Sel(\EC/\Q_\cyc)\oplus \Sel(\EC^{(K)}/\Q_\cyc)
\end{equation}
of Proposition~\ref{prop:selmer-decomposition-ord} allows us to apply Proposition~\ref{prop:SC-ord} and verify hypothesis $\eqref{hyp:SC}$.
Conditions (3) and (4) imply
\[
1\leq \ord_TL_p(\EC)+\ord_TL_p(\EC^{(K)})\leq \lambda\big(L_p(\EC)\big)+\lambda\big(L_p(\EC^{(K)})\big)=1.
\]
Thus $\ord_TL_p(\EC)+\ord_TL_p(\EC^{(K)})=1$.
Now \eqref{eq:selmer-decomp-ord} yields
\[
\ord_Y\left(\Char_{\Lambda_\cyc}\left(\Sel(\EC/K_\cyc)^\vee\right)\right)=\ord_TL_p(\EC)+\ord_TL_p(\EC^{(K)})=1.
\]
 It then follows from Corollary~\ref{cor:main-result-ord} that \ref{Mazur-conj} holds for $(\EC,K,p)$.

\begin{remark}
In view of standard conjectures in number theory, it seems reasonable to expect that conditions $(0)$--$(4)$ are satisfied over infinitely many $K$ for a fixed pair $(\EC,p)$.
It will be interesting to work out the precise proportions in this case and also in the cases when $(p,K)$ is fixed but $\EC$ varies or when $(\EC,K)$ is fixed and $p$ varies.
But this is not the focus of our current investigation.
\end{remark}

\begin{example}[rank 0]
Let $\EC$ be the elliptic curve with Cremona label \href{http://www.lmfdb.org/EllipticCurve/Q/11a1}{\texttt{11a1}}, let $K=\Q(\sqrt{-13})$ and $p=7$.
Note that $\mathrm{rank}_\Z(\EC(\Q))=0$.
We check that $\EC^{(K)}$, which has Cremona label \href{http://www.lmfdb.org/EllipticCurve/Q/29744z2}{\texttt{29744z2}}, has root number $-1$ and that the pair $(K,N_{\EC})$ satisfies (\ref{hyp: GHH}).
Furthermore, $7$ splits in $K$, and $a_7(\EC/\Q)=-2$ with $\EC(K)[7]=0$.
Using Magma, we compute the $7$-adic $L$-functions of $\EC$ and $\EC^{(K)}$ as follows:
\begin{align*}
L_7(\EC)&= -1827287233098838071872685754  + 1188T -120T^2 +882T^3 +991T^4 + 136T^5 + O(T^6), \\
L_7(\EC^{(K)})&=213T -649T^2-1190T^3 -974T^4 + 101T^5 + O(T^6).
\end{align*}
In agreement with what these calculations suggest, the LMFDB \cite{lmfdb} confirms that the $\mu$-invariants of both $L_7(\EC)$ and $L_7(\EC^{(K)})$ vanish and
\[
\lambda\big(L_7(\EC)\big)=0 \quad \text{and} \quad
\lambda\big(L_7(\EC^{(K)})\big) =1.
\]
In light of Remark \ref{Lp-prec}, these computations also show that
\[
0 = \ord_TL_7(\EC)  \quad \text{and} \quad
1 \leq \ord_TL_7(\EC^{(K)}).
\]
Hence conditions (0)--(4) are satisfied.
Thus, \ref{Mazur-conj} holds for the triple $(\texttt{11a1},\Q(\sqrt{-13}),7)$.
\end{example}

\begin{example}[rank 1]
Let $\EC$ be the elliptic curve with Cremona label \href{http://www.lmfdb.org/EllipticCurve/Q/37a1}{\texttt{37a1}}, let $K=\Q(\sqrt{-11})$, let $p=5$, and let $\EC^{(K)}$ be the corresponding twist of $E$, which has Cremona label \href{http://www.lmfdb.org/EllipticCurve/Q/4477a1}{\texttt{4477a1}}.
Note that $\mathrm{rank}_\Z(\EC(\Q))=1$.
Furthermore, $\EC^{(K)}$ has root number $-1$, and the pair $(K,N_{\EC})$ satisfies (\ref{hyp: GHH}).
We have that $5$ splits in $K$, $a_5(\EC)=-2$, and $\EC(K)[7]=0$.
Using the LMFDB, we find that the $\mu$-invariants of $L_5(\EC)$ and $L_5(\EC^{(K)})$ vanish and
\[
\lambda\big(L_5(\EC)\big)=1 \quad \text{and} \quad
\lambda\big(L_5(\EC^{(K)})\big) =0.
\]
Using Magma, we check (as in the previous example) that
\[
1 \leq \ord_TL_5(\EC)  \quad \text{and} \quad
0 = \ord_TL_5(\EC^{(K)}).
\]
Hence conditions (0)--(4) are satisfied.
Thus, \ref{Mazur-conj} holds for the triple $(\texttt{37a1},\Q(\sqrt{-11}),5)$.
\end{example}

\subsection{Supersingular reduction}
\label{sec:examples-ss}

As in the ordinary case, we list sufficient conditions needed to verify the hypotheses of Corollary~\ref{cor:main-result-ss}.
It suffices to check that
\begin{enumerate}
\item $\EC/\Q$ has good reduction at $p$,
\item $a_p(\EC/\Q)=0$,
\item $p$ splits in $K$,
\item the two primes above $p$ are totally ramified in $K_\ac$,  and
\item $\Sel(\EC/K)$ is finite.
\end{enumerate}
To check the last condition, it is enough to have $L(\EC/K,1)\ne0$.
Before explaining this, let us recall a useful decomposition result, which is analogous to Proposition~\ref{prop:selmer-decomposition-ord}.

\begin{proposition}
\label{prop:selmer-decomposition-ss}
Let $\EC/\Q$ be an elliptic curve with good \emph{supersingular} reduction at $p$, and let $\EC^{(K)}$ denote the twist of $E$ corresponding to $K$.
Let $\heartsuit \in \{+,-\}$.
We have the decomposition
\[
\Sel^{\hh}(\EC/K_\cyc) \simeq \Sel^{\heartsuit}(\EC/\Q_\cyc) \oplus \Sel^{\heartsuit}(\EC^{(K)}/\Q_\cyc),
\]
where $\Sel^\heartsuit(\EC/\Q_\cyc)$ and $\Sel^\heartsuit(\EC^{(K)}/\Q_\cyc)$ are Kobayashi's $\pm$-Selmer groups defined in \cite{kobayashi03}.
\end{proposition}

\begin{proof}
See the arguments given in \cite[proof of Proposition 8.4]{LP}.
\end{proof}

Now, assuming $L(\EC/K,1)\ne0$, we deduce from the interpolation formulae of Pollack's plus and minus $L$-functions that $L_p^\pm(\EC)$ and $L_p^\pm(\EC^{(K)})$ do not vanish at the trivial character (see \cite[(3.6)]{kobayashi03}).
Then Theorem~\cite[Theorem~4.1]{kobayashi03} together with Proposition~\ref{prop:selmer-decomposition-ss} imply that, up to a power of $p$,
\[
L_p^\heartsuit(\EC)L_p^\heartsuit(\EC^{(K)})\in \Char_{\Lambda_\cyc}\left(\Sel^\hh(\EC/K_\cyc)^\vee\right).
\]
 In particular, $\Sel^\hh(\EC/K_\cyc)^{\Gamma_\cyc}$ is finite.
 As we have seen in the proof of Proposition~\ref{prop:torsion-finite}, there is an injection $\Sel(\EC/K)\hookrightarrow\Sel^\hh(\EC/K_\cyc)^{\Gamma_\cyc}$.
 Therefore, $\Sel(\EC/K)$ is finite.

\begin{example}Let $\EC/\Q$ be the elliptic curve with Cremona label \href{https://www.lmfdb.org/EllipticCurve/Q/14a1/}{\texttt{14a1}}.
Each of the sufficient conditions is easy to check using a computer algebra system like Magma (code available upon request).
In particular, for primes $d < 100$ and supersingular primes $p<100$, we can use Corollary~\ref{cor:main-result-ss} to verify \ref{Mazur-conj} for $(\EC,\Q(\sqrt{-d}),p)$ for each of the following pairs $(p,d)$:
\begin{table}[H]
\begin{tabular}{|l|llll|}
\hline
$p$ &  &      & $d$ &     \\
\hline
5   &  & (19, & 59, & 71) \\
\hline
11  &  & (19, & 73, & 79) \\
\hline
23  &  & (19, & 79, & 83) \\
\hline
71  &  & (23, & 59) &  \\
\hline
\end{tabular}
\end{table}
\end{example}

\begin{example}Let $\EC/\Q$ be the elliptic curve with Cremona label \href{https://www.lmfdb.org/EllipticCurve/Q/30a1/}{\texttt{30a1}}.
For primes $d < 100$ and supersingular primes $p<100$, we can use Corollary~\ref{cor:main-result-ss} to verify \ref{Mazur-conj} for $(\EC,\Q(\sqrt{-d}),p)$ for each of the following pairs $(p,d)$:
\begin{table}[H]
\begin{tabular}{|l|llllllll|}
\hline
$p$ &  &  &      &     & \multicolumn{2}{l}{$d$} &     &     \\
\hline
11  &  &  & (43, & 79) &            &            &     &     \\
\hline
23  &  &  & (11, & 43, & 67,        & 79)        &     &     \\
\hline
47  &  &  & (11, & 23, & 31,        & 43,        & 67) &     \\
\hline
59  &  &  & (11, & 23  & 31,        & 43,        & 47, & 67) \\
\hline
71  &  &  & (11, & 23, & 31,        & 47,        & 59, & 67)\\
\hline
\end{tabular}
\end{table}
\end{example}

\appendix
\section{Signed Selmer Groups in Anti-Cyclotomic Extensions}
\label{appendix}

We investigate properties of signed Selmer groups over the anti-cyclotomic $\Zp$-extension.
As in \S~\ref{sec: Cyclotomic - ss case}, we fix an elliptic curve $\EC/\Q$ which has good supersingular reduction at $p$ with $a_p(\EC/\Q)=0$.
Throughout, we assume that $p$ splits in the imaginary quadratic field $K$ and that both primes above $p$ are totally ramified in $K_\ac/K$.

\subsection{Cotorsionness of signed Selmer groups}

We begin with the following control theorem, which is analogous to Proposition~\ref{prop:control-ord} in the ordinary case.

\begin{proposition}[Control Theorem - the supersingular case]
\label{prop:signed-control}
Let $\EC/\Q$ be an elliptic curve and let $p$ be an odd prime of good \emph{supersingular} reduction with $a_p(\EC/\Q)=0$.
Let $\cK=K_{a,b}$ for some $(a \colon b) \in \mathbb{P}^1(\Zp)$, and let $H=H_{a,b}$.
For $\bullet, \star\in\{+,-\}$, the natural restriction map induces an isomorphism
\[
\Sel^\bc(\EC/\cK)\simeq \Sel^\bc(\EC/K_\infty)^{H}.
\]
\end{proposition}

\begin{proof}
Recall from \cite[Lemma~3.4]{lei2021akashi} that $\EC(K_\infty)[p^\infty]=0$.
The proof of the proposition is the same as that of Proposition~\ref{prop:control-ord}, except that, for $v \mid p$, we must show the injectivity of the restriction
\[
\frac{H^1(K_{a,b,\fr},\Ep)}{H^1_{\bullet}(K_{a,b,\fr},\Ep)}\longrightarrow \frac{H^1(K_{\infty,\fr},\Ep)}{H^1_{\star}(K_{\infty,\fr},\Ep)}
\]
for $\fr\in\{\fp,\fq\}$ {and $\bullet,\star \in \{+,-\}$}, which follows from the same argument as the one given in \cite[Proposition 5.8]{leisujatha}.
\end{proof}
This generalizes \cite[Proposition~4.7]{parham} where the case $\cK=K_\cyc$ has been studied.
Proposition~\ref{prop:signed-control} immediately implies the next result.

\begin{corollary}
\label{cor:torsion-control}
Let $\bullet,\star\in\{+,-\}$ and $(a \colon b) \in \mathbb{P}^1(\Zp)$ such that $\Sel^\bc(\EC/K_{a,b})^\vee$ is torsion over $\Lambda_{a,b}$.
Then $\Sel^\bc(\EC/K_\infty)^\vee$
is torsion over $\Lambda$.
\end{corollary}

We will need the following result for the cyclotomic $\Zp$-extension of $K$.

\begin{lemma}
\label{lem:ss-cyclo-torsion}
The $\Lambda_\cyc$-modules $\Sel^{++}(\EC/K_\cyc)^\vee$ and $\Sel^{--}(\EC/K_\cyc)^\vee$ are torsion.
If furthermore $\Sel(\EC/K)$ is finite, then $\Sel^{+-}(\EC/K_\cyc)^\vee$ and $\Sel^{-+}(\EC/K_\cyc)^\vee$ are also $\Lambda_\cyc$-torsion.
In this case,
\[
\ord_Y\left(\Char_{\Lambda_\cyc}\left(\Sel^{\bb \star}(\EC/K_\cyc)^\vee\right)\right)=0 \text{ for all }\bb,\star \in \{+,-\}.
\]
\end{lemma}

\begin{proof}
It follows from \cite[Theorem~1.2]{kobayashi03} and \cite[Proposition~8.4]{LP} that both
$\Sel^{++}(\EC/K_\cyc)^\vee$ and $\Sel^{--}(\EC/K_\cyc)^\vee$ are $\Lambda_\cyc$-torsion.

The rest of the lemma follows from the proof of Proposition~\ref{prop:torsion-finite}, which tells us that there is an injection $\Sel(\EC/K)\rightarrow\Sel^\bc(\EC/K_\cyc)^{\Gamma_\cyc}$ with finite cokernel.
\end{proof}

\begin{remark}\leavevmode
\label{rk:+-torsion}
\begin{enumerate}
\item[i)] We utilize the assumption that $a_p(\EC/\Q)=0$ when we invoke \cite[Theorem~1.2]{kobayashi03} and \cite[Proposition~8.4]{LP}.
\item[ii)] On combining Lemma~\ref{lem:ss-cyclo-torsion} with Corollary~\ref{cor:torsion-control}, we see that $\Sel^{++}(\EC/K_\infty)^\vee$ and $\Sel^{--}(\EC/K_\infty)^\vee$ are \textbf{always} torsion over $\Lambda$.
If in addition $\Sel(\EC/K)$ is finite, then the two mixed signed Selmer groups over $K_\infty$ are also torsion over $\Lambda$.
\end{enumerate}
\end{remark}

\subsection{Pseudo-null submodules and signed Selmer groups}
The main goal of this section is to prove a supersingular analogue of  \cite[Proposition~6.1]{KLR}, regarding $M_o=0$ when $M$ is the Pontryagin dual of a signed Selmer group of $\EC$ over $K_\infty$.

\begin{proposition}
\label{prop:ss-fg-m0}
Let $\EC/\Q$ be an elliptic curve with good \emph{supersingular} reduction at $p$ and $a_p(\EC/\Q)=0$.
For $\heartsuit\in\{+,-\}$, the $\Lambda$-module $\Sel^{\hh}(\EC/K_\infty)^\vee$ admits no non-trivial pseudo-null submodules.
\end{proposition}

\begin{proof}
This proposition is a special case of \cite[Proposition~4.1.1]{Greenberg16} and its proof simplifies considerably as in \cite[proof of Proposition~6.1]{KLR}.
We repeat the proof for the reader's convenience and follow Greenberg's notation as closely as possible.

Set $\cT=T_p(\EC)\otimes \Lambda^\iota$, where as usual $T_p(\EC)$ is the $p$-adic Tate module of $\EC$ and $\iota$ is the involution on $\Lambda$ sending a group-like element to its inverse.
Further, write $\cD=\cT\otimes_{\Lambda} \Lambda^\vee$.

\vspace{0.2cm}

\noindent \textbf{RFX($\cD$):} $\cT$ is a reflexive $\Lambda$-module.

\noindent \emph{Justification:} This holds since it is free module over $\Lambda$.

\vspace{0.2cm}

\noindent \textbf{LEO($\cD$):} The module
\[
\ker\left(H^2(K_{\Sigma}/K,\cD)\longrightarrow \prod_{v\in \Sigma}H^2(K_v,\cD)\right)
\]
is $\Lambda$-cotorsion. 

\noindent \emph{Justification:}
First, we recall from \cite[Theorem~3]{greenberg06} that there is an isomorphism of $\Lambda$-modules $H^2(K_{\Sigma}/K,\cD)\simeq H^2(K_{\Sigma}/K_\infty,\EC[p^\infty])$.

Recall from Lemma~\ref{lem:ss-cyclo-torsion}
that $\Sel^{\hh}(\EC/K_\cyc)^\vee$ is $\Lambda_\cyc$-torsion.
We can refer to the argument in \cite[Remark~2.2]{KLR} to conclude that $\Sel^{\hh}(\EC/K_\infty)^\vee$ is a finitely generated torsion $\Lambda$-module.

Next, we remind the reader that \cite[Theorems~3.2]{OV03} asserts that
\[
\rank_{\Lambda}H^1(K_\Sigma/K_\infty, \EC[p^\infty]) - \rank_{\Lambda}H^2(K_\Sigma/K_\infty, \EC [p^\infty]) = 2.
\]
A standard argument with Poitou--Tate exact sequence then tells us that $H^2(K_{\Sigma}/K_\infty,\EC [p^\infty])$ is $\Lambda$-cotorsion (see for example \cite[Proposition~4.12]{lei2021akashi}), whereas $H^1(K_\Sigma/K_\infty, \EC[p^\infty])$ has corank two.
In particular, \textbf{LEO($\cD$)} holds.

\vspace{0.2cm}

\noindent \textbf{CRK}$(\cD,\cL)$: The following equality holds\footnote{Here, we have rephrased the condition \textbf{CRK}$(\cD,\cL)$ in \cite[\S2.3]{greenberg10} using \cite[Theorem~3]{greenberg06}.}
\begin{align*}
\corank_\Lambda H^1(K_{\Sigma}/K_\infty, \EC[p^\infty])& =\corank_\Lambda \Sel^{\hh}(\EC/K_\infty) + \corank_\Lambda \bigoplus_{\substack{v\mid p\\v\in \Sigma(F_\infty)}}J_v^{\hh}(\EC/K_\infty) \\
& \qquad + \corank_\Lambda \bigoplus_{\substack{v\nmid p\\v\in \Sigma(K_\infty)}}J_v(\EC/K_\infty)
\end{align*}
and both sides equal to $2$.


\noindent \emph{Justification:}
In view of the above discussion, we only need to show that the $\Lambda$-corank of the local cohomology groups is equal to 2.
This assertion is proven in \cite[p.~1012]{lei2021akashi}).

Set $\cT^*=\Hom(\cD,\mu_{p^\infty})$ and as before write $G=\Gal(K_\infty/K)$.

\vspace{0.2cm}

\noindent \textbf{LOC}$_v^{(1)}(\cD)$\textbf{:}
$(\cT^*)^{G_{K_v}}=0$ for each $v\in \Sigma$. \newline
\noindent \textbf{LOC}$_v^{(2)}(\cD)$\textbf{:}
$\cT^*/(\cT^*)^{G_{K_v}}$ is a reflexive $\Lambda$-module.

\noindent \emph{Justification:}
Since $p\ne2$, we have $(\cT^*)^{G_{K_v}}=0$ when $v$ is an archimedean prime.
Furthermore, if $v$ is a non-archimedean prime, it does not split completely in $K_\infty$.
By \cite[Lemma~5.2.2]{greenberg10}, we know that $(\cT^*)^{G_{K_v}}=0$.
As $\cT^*$ is a free $\Lambda$-module, the conditions \textbf{LOC}$_v^{(1)}(\cD)$ and \textbf{LOC}$_v^{(2)}(\cD)$ both hold for all $v\in \Sigma$.

Finally, we note that the condition $\cD[\mathfrak{m}]$ admits no quotient isomorphic to $\mu_p$ for the action of $G_K$ (assumption (b) in \cite[Proposition~4.1.1]{Greenberg16}) is equivalent to $\EC(K)[p]=0$ via the Weil pairing (see the last paragraph on p.~248 of \emph{op. cit.}).
The latter assumption is automatically satisfied since $p$ is a prime of supersingular reduction.
The result is a direct consequence of \cite[Proposition~4.1.1]{Greenberg16}.
\end{proof}

\begin{remark}
\label{rk:mixed-m0}
Let $\bc\in\{+-,-+\}$.
If in addition to the hypothesis of Proposition~\ref{prop:ss-fg-m0}, we also know that $\Sel^\bc(E/K_\infty)^\vee$ is torsion over $\Lambda$ (e.g., when $\Sel(\EC/K)$ is finite as in Lemma~\ref{lem:ss-cyclo-torsion}), the proof of Proposition~\ref{prop:ss-fg-m0} goes through verbatim and we can conclude that  $\Sel^{\bb \star}(\EC/K_\infty)^\vee$ do not admit any non-trivial pseudo-null submodules.
\end{remark}

\begin{remark}
The (non)existence of pseudonull submodules for $\Sel^{\hh}(\EC/K_\infty)^\vee$ has also been studied (under slightly different hypotheses) in \cite[Theorem~5.6]{parham} via a different method.
\end{remark}

\begin{corollary}
\label{cor:project-ss}
Let $\EC/\Q$ be an elliptic curve with good \emph{supersingular} reduction at $p$ satisfying $a_p(\EC/\Q)=0$.
Let $\bullet,\star\in\{+,-\}$ and $(a \colon b) \in \mathbb{P}^1(\Zp)$ such that $\Sel^\bc(\EC/K_\infty)^\vee$ satisfies \eqref{hyp:SC} and that $\Sel^\bc(\EC/K_{a,b})^\vee$ is torsion over $\Lambda_{a,b}$, then the following equality holds:
\[
\pi_{a,b}(\Char_\Lambda(\Sel^{\bb \star}(\EC/K_\infty)^\vee))=\Char_{\Lambda_{a,b}}(\Sel^{\bb \star}(\EC/K_{a,b})^\vee).
\]
\end{corollary}

\begin{proof}
By Remark~\ref{rk:+-torsion} the assumption on the $\Lambda_{a,b}$-torsionness of $\Sel^{\bb\star}(\EC/K_{a,b})^\vee$ implies that the $\Lambda$-module $M:= \Sel^{\bb \star}(\EC/K_\infty)^\vee$ is torsion.
Furthermore,  Proposition~\ref{prop:ss-fg-m0} and Remark~\ref{rk:mixed-m0} tell us that $M_o=0$.
Therefore, the asserted equality follows from Theorem~\ref{thm:structure} and Proposition~\ref{prop:signed-control}.
\end{proof}

\begin{definition}
For $\bullet, \star \in\{+,-\}$, define $\sT(\EC,K)^{\bullet \star}$ to be the set of $(a:b)\in\PP^1(\Zp)$ such that $\Sel^{\bb \star}(\EC/K_{a,b})^\vee$ is \emph{not} a torsion $\Zp\llbracket\Gamma_{a,b}\rrbracket$-module and define $n(\EC,K)^{\bullet \star} = \#\sT(\EC,K)^{\bullet \star}$.

Finally, set
\[
\sT(\EC,K)= \bigcup_{\bb,\star \in \{+,-\}}\sT(\EC,K)^{\bullet \star}.
\]
\end{definition}

In view of the above definition, Definition~\ref{def:growth-number}, and Corollary \ref{cor:ss-bounded-torsion} we can now conclude that $\sM(\EC,K) \subset \sT(\EC,K)$ and
\[
n(\EC,K) \leq \# \sT(\EC,K) \leq  \sum_{\bb,\star \in \{+,-\}} n(\EC,K)^{\bullet \star}.
\]

\begin{proposition}
\label{prop:count-ss}
For $\bullet, \star\in\{+,-\}$, if $\Sel^{\bb\star}(\EC/K_\cyc)^\vee$ is torsion over $\Lambda_\cyc$ and $\Sel^{\bb\star}(\EC/K_\infty)^\vee$ satisfies \eqref{hyp:SC},  the following inequality holds
\[
n(\EC,K)^{\bullet\star}\le\ord_Y\left(\Char_{\Lambda_\cyc}\left(\Sel^{\bb\star}(\EC/K_\cyc)^\vee\right)\right).
\]
\end{proposition}

\begin{proof}
In light of Corollary~\ref{cor:project-ss}, this follows from the same proof as that of Proposition~\ref{prop:ord-count}.
\end{proof}

\subsection{Sufficient conditions for \texorpdfstring{\eqref{hyp:SC}}{}}
In this section, we prove the supersingular analogue of Proposition~\ref{prop:SC-ord}.

\begin{proposition}
\label{prop:mismatched-lambda}
Let $\EC/\Q$ be an elliptic curve and $p$ be a prime of good \emph{supersingular} reduction such that $a_p(\EC/\Q)=0$.
Suppose furthermore that $p$ does not divide the Tamagawa number of $\EC/K$.
Then $\Sel^{\bc}(\EC/K_\infty)^\vee$ is cyclic over $\Lambda$ for one choice of $\bc$ if and only if $\Sel^\bc(\EC/K_\infty)^\vee$ is cyclic over $\Lambda$ for all four choices of $\bc$.
\end{proposition}

\begin{proof}
As in the proof of Proposition~\ref{prop:torsion-finite}, we have an injection
\[
\Sel(\EC/K)\hookrightarrow \Sel^{\bc}(\EC/K_\infty)^{G_\infty}
\]
whose cokernel is isomorphic to $\prod_{v\nmid p}\ker\gamma_v^\bc$. It follows from \cite[proof of Lemma~3.3]{greenberg99} that this cokernel is trivial under our hypothesis on the Tamagawa number of $\EC/K$.
Therefore, by Nakayama's Lemma, $\Sel^{\bc}(\EC/K_\infty)^\vee$ is a cyclic $\Lambda$-module if and only if $\Sel(\EC/K)^\vee$ is a cyclic $\Zp$-module (which is independent of the choice of $\bc$).
\end{proof}

\begin{corollary}
\label{cor:sc-satisfied}
Let $\EC$, $K$ and $p$ be as in Proposition~\ref{prop:mismatched-lambda}.
If there exists $\heartsuit\in\{+,-\}$ such that $\lambda(\Sel^{\hh}(\EC/K_\cyc))\le 1$ and $\mu(\Sel^{\hh}(\EC/K_\cyc))=0$, then $\Sel^\bc(\EC/K_\infty)^\vee$ satisfies \eqref{hyp:SC} for all four choices of $\bc$.
\end{corollary}

\begin{proof}
By \cite[proof of Theorem 3.14]{kim13}, $\Sel^{\hh}(\EC/K_\cyc)$ has no proper $\Lambda_\cyc$-submodules of finite index.
Thus if $\lambda(\Sel^{\hh}(\EC/K_\cyc)\le 1$ and $\mu(\Sel^{\hh}(\EC/K_\cyc))=0$, it follows that $\Sel^{\hh}(\EC/K_\cyc) \simeq \Qp/\Zp$ or $0$.
In particular, it is co-cyclic over $\Lambda_\cyc$.
Now, it follows from Proposition~\ref{cor:main-result-ss} and Nakayama's lemma that $\Sel^{\hh}(\EC/K_\infty)^\vee$ is cyclic over $\Lambda$.
The same holds for all four choices of $\bc$ by Proposition~\ref{prop:mismatched-lambda}.
In particular, \eqref{hyp:SC} holds.
\end{proof}

\subsection{Signed Selmer groups under \texorpdfstring{\eqref{hyp: GHH}}{}}
In this section, we study the signed Selmer groups under the modified Heegner hypothesis.

\begin{theorem}
\label{thm:signed-Selmer-anticyclo}
Suppose that \eqref{hyp: GHH} is satisfied for $(K,N_{\EC})$.
Then the $\Lambda_{\ac}$-modules $\Sel^{++}(\EC/K_\ac)^\vee$ and $\Sel^{--}(\EC/K_\ac)^\vee$ are of rank one, whereas $\Sel^{+-}(\EC/K_\ac)^\vee$ and $\Sel^{-+}(\EC/K_\ac)^\vee$ are $\Lambda_\ac$-torsion.
\end{theorem}

\begin{proof}
The assertion on the $++$ and $--$ Selmer groups has been proved in \cite[Theorem~1.4]{LongoVigni2} under the modified Heegner hypothesis.
Let $\bullet\in\{+,-\}$.
Recall\footnote{\cite[Lemma~6.7 and Theorem~6.8]{castellawan} are proven under the general Heegner hypothesis which is slightly weaker than \eqref{hyp: GHH} since it allows $N^+$ to include primes which are ramified in $K$.} that \cite[Lemma~6.7]{castellawan} asserts\footnote{See \cite{castellawan} for the definitions of the local conditions $0$ and $\emptyset$, where they are called rel and str, respectively.
In short, these are the weakest and strongest local conditions that one may define at these places.}
\[
\rank_{\Lambda_\ac}\Sel^{\emptyset \bullet}(\EC/K_\ac)^\vee=1+\rank_{\Lambda_\ac}\Sel^{\bullet0}(\EC/K_\ac)^\vee.
\]
Since $\Sel^{\emptyset0}(\EC/K_\ac)^\vee$ is $\Lambda_\ac$-torsion by \cite[Theorem~6.8]{castellawan} 
and that it contains $\Sel^{\bullet0}(\EC/K_\ac)$, it follows that
\[
\rank_{\Lambda_\ac}\Sel^{\emptyset \bullet}(\EC/K_\ac)^\vee=1,\quad \rank_{\Lambda_\ac}\Sel^{\bullet0}(\EC/K_\ac)^\vee=0.
\]
On taking complex conjugation,  we deduce that
\[
\rank_{\Lambda_\ac}\Sel^{ \bullet\emptyset}(\EC/K_\ac)^\vee=\rank_{\Lambda_\ac}\Sel^{\emptyset \bullet}(\EC/K_\ac)^\vee=1.
\]
Furthermore, we have the following defining exact sequence\footnote{See \cite{leisprung} for the definitions of the local condition $1$.} :
\[
0\longrightarrow \Sel^{\bullet 0}(\EC/K_\ac)\longrightarrow \Sel^{\bullet 1}(\EC/K_\ac)\longrightarrow \EC(K_\fq)\otimes\Qp/\Zp.
\]
As both $(\EC(K_\fq)\otimes\Qp/\Zp)^\vee$ and $\Sel^{\bullet 0}(\EC/K_\ac)$ are $\Lambda_\ac$-torsion, it follows that     $\Sel^{\bullet 1}(\EC/K_\ac)^\vee$  is also $\Lambda_\ac$-torsion.

We now consider the exact sequence given as in \cite[Proposition~4.3]{leisprung}:
\[
0\longrightarrow \Sel^{\bullet1}(\EC/K_\ac)\longrightarrow  \Sel^{\bullet\bullet}(\EC/K_\ac)\oplus  \Sel^{\bullet\star}(\EC/K_\ac)\longrightarrow  \Sel^{\bullet\emptyset}(\EC/K_\ac),
\]
where $\star\in\{+,-\}\setminus\{\bullet\}$.
We have seen that the $\Lambda_\ac$-coranks of the first term and the third term of this exact sequence are  equal to $0$ and $1$, respectively.
Since $\Sel^{\bullet\bullet}(\EC/K_\ac)^\vee$ is of rank one, it follows that the $\Lambda_\ac$-rank of $\Sel^{\bullet\star}(\EC/K_\ac)^\vee$ is zero, which proves the second assertion of the theorem.
\end{proof}

\begin{corollary}
Under the same hypothesis as Theorem~\ref{thm:signed-Selmer-anticyclo}, the $\Lambda$-modules  $\Sel^{+-}(\EC/K_\infty)^\vee$ and $\Sel^{-+}(\EC/K_\infty)^\vee$ are torsion.
\end{corollary}

\begin{proof}
This follows by combining Proposition~\ref{prop:signed-control} with the second assertion of Theorem~\ref{thm:signed-Selmer-anticyclo}.
\end{proof}

The torsionness of $\Sel^{+-}(\EC/K_\ac)^\vee$ and $\Sel^{+-}(\EC/K_\ac)^\vee$ means that a supersingular analogue of the second assertion of Corollary~\ref{cor:main-result-ord} is not available, at least not via the strategy employed in the main body of the article.
Indeed, under \eqref{hyp: GHH}, we would have
\[
\ord_Y\left(\Char_{\Lambda_\cyc}\Sel^\bc(\EC/K_\cyc)\right)\ge1
\]
for all four choices of $\bc$.
If equality holds, since $\Sel^{+-}(\EC/K_\ac)^\vee$ and $\Sel^{+-}(\EC/K_\ac)^\vee$ are torsion $\Lambda_\ac$, we cannot rule out that $\Sel^{+-}(\EC/K_{a,b})^\vee$ or $\Sel^{+-}(\EC/K_{a,b})^\vee$ might have positive rank over $\Lambda_{a,b}$ for some $(a:b)\in\PP^1(\Zp)$.

Despite this shortcoming, we may still find examples where the signed Selmer groups over $K_\infty$ are co-cyclic, answering Hida's question in new settings.
We build on the following supersingular extension of
Theorem~\ref{thm: matar-nekovar 4.8} proved by A.~Matar \cite{matar21}.

\begin{theorem}\label{thm:Matar}
Let $\EC/\Q$ be an elliptic curve of conductor $N$ and let $K$ be an imaginary quadratic field such that every prime $\ell \mid N$ splits in $K$.
Let $p$ be an odd prime of good \emph{supersingular} reduction of $\EC$.
Further suppose that
\begin{itemize}
\item $p$ splits in $K$ and both primes above $p$ in $K$ are totally ramified in $K_{\ac}/K$.
\item $p$ does not divide the Tamagawa number of $\EC/\Q$.
\item the basic Heegner point $y_K\notin p\EC(K)$.
\end{itemize}
Then, $\Sel^{++}(\EC/K_{\ac})^\vee$ and $\Sel^{--}(\EC/K_{\ac})^\vee$ are both free modules with $\Lambda_{\ac}$-rank equal to 1.
\end{theorem}

\begin{proof}
See \cite[Theorem~3.1]{matar21}.
In \S~2, the author of \emph{op. cit.} notes that $p\geq 5$ throughout the article.
However, the proof of Theorem~3.1 in \emph{op. cit.} does not require the assumption because the proof depends only on Theorem~2.2 in \emph{op. cit.} which holds for all odd $p$.
\end{proof}

\begin{remark}
In the setting of Theorem~\ref{thm:Matar}, $\Sel^{\hh}(\EC/K_{\ac})^\vee$ is a free module with $\Lambda_{\ac}$-rank equal to 1. In particular, it is cyclic over $\Lambda_\ac$ and $\Sel^{\hh}(\EC/K_{\infty})^\vee$ satisfies \eqref{hyp:SC}  (by Nakayama's lemma and Proposition~\ref{prop:signed-control}).
Then using Proposition~\ref{prop:mismatched-lambda} we can conclude that all four signed Selmer groups $\Sel^{\bc}(\EC/K_{\infty})^\vee$ satisfy \eqref{hyp:SC}
\end{remark}

\subsection{Verifying Hypothesis~\texorpdfstring{\eqref{hyp:SC}}{} in the supersingular setting}

If $\EC/\Q$ has good \textit{supersingular} reduction at $p$ with $a_p(\EC)=0$ and $K$ is a fixed imaginary quadratic field, let $L_p^{\pm}(\EC),\; L_p^{\pm}(\EC^{(K)})\in\Z_p\llbracket T\rrbracket$ denote the signed $p$-adic $L$-functions attached to $\EC/\Q$ and $\EC^{(K)}/\Q$ of Pollack \cite{pollack03}.
Let $L_p(\EC,\alpha)\in \Q_p(\alpha)\llbracket T\rrbracket$ denote the unbounded $p$-adic $L$-function of $\EC/\Q$ at a fixed root $\alpha$ of the polynomial $x^2+p$.
We now explain how we can verify hypothesis \eqref{hyp:SC} in the supersingular setting using Corollary~\ref{cor:sc-satisfied} where we recall that we assumed \eqref{hyp: GHH} rather than the strict Heegner hypothesis required in Theorem~\ref{thm:Matar}.

Conditions (1) and (2) of $\S~$\ref{sec:examples-ord} remain in the supersingular setting; however, we may ignore condition (4), and we replace conditions $(0)$ and $(3)$ with
\begin{enumerate}[label={(\arabic*$'$)}]
\setcounter{enumi}{-1}
\item $a_p(\EC/\Q)=0$, $p$ splits in $K$, and both primes of $K$ lying above $p$ are totally ramified in $K_\ac$.
\end{enumerate}
\begin{enumerate}[label={(\arabic*$'$)}]
\setcounter{enumi}{2}
\item $\lambda\big(L_p^\heartsuit(\EC)\big)+\lambda\big(L_p^\heartsuit(\EC^{(K)})\big)=1$ and $\mu\big(L_p^\heartsuit(\EC)\big)=\mu\big(L_p^\heartsuit(\EC^{(K)})\big)=0$ for (at least) one choice of $\heartsuit \in \{+,-\}$.
\end{enumerate}
We need to check one further condition to use Corollary~\ref{cor:sc-satisfied}:
\begin{enumerate}[label={(\arabic*$'$)}]
\setcounter{enumi}{4}
\item $p$ does not divide the Tamagawa number of $\EC/K$.
\end{enumerate}

Let us explain condition (3$'$) in more detail.
Recall the isomorphism from Proposition~\ref{prop:selmer-decomposition-ss}, i.e.,
\[
\Sel^{\hh}(\EC/K_\cyc) \simeq \Sel^{\heartsuit}(\EC/\Q_\cyc) \oplus \Sel^{\heartsuit}(\EC^{(K)}/\Q_\cyc).
\]
As in the ordinary case, we have
\[
\lambda\left(\Sel^\hh(\EC/K_\cyc)^\vee\right)= \lambda\left(\Sel^\heartsuit(\EC/\Q_\cyc)^\vee\right) + \lambda\left(\Sel^\heartsuit(\EC^{(K)}/\Q_\cyc)^\vee\right),
\]
and similarly for the $\mu$-invariants.
Thus, it suffices to study the Iwasawa invariants over $\Q_\cyc$.

For each of our examples, the \textit{analytic} Iwasawa invariants $\lambda(L_p^{{\heartsuit}}(\EC))$ and $\mu(L_{{p}}^{{\heartsuit}}(\EC))$ are available on the LMFDB where $\heartsuit \in \{ +,-\}$.
Thus, we need to ensure that we can make deductions about the algebraic Iwasawa invariants using this data.

On the one hand, by \cite[Theorem 1.3]{kobayashi03} we have
\[
\lambda\left(\Sel^{\hh}(\EC/K_\cyc)^\vee\right) \leq \lambda\big(L_p^\heartsuit(\EC)\big),
\]
and similarly for the $\mu$-invariants.
Therefore, verifying condition (3$'$) allows us to deduce
\[
\lambda\left(\Sel^{\hh}(\EC/K_\cyc)^\vee\right) \leq 1 \quad \text{and} \quad \mu\left(\Sel^{\hh}(\EC/K_\cyc)^\vee\right) =0
\]
for some $\heartsuit \in \{+,-\}$.
On the other hand, the parity conjecture \cite[Theorem 4.29]{kim_parity} and our assumption that $\EC/K$ has root number $-1$ implies
\[
\mathrm{corank}_{\Zp}(\Sel^{\hh}(\EC/K)) \geq 1.
\]
As we have seen in the proof of Proposition~\ref{prop:torsion-finite}, there is an injection $\Sel(\EC/K)\hookrightarrow\Sel^\hh(\EC/K_\cyc)^{\Gamma_\cyc}$. It then follows 
that
\begin{equation*}\label{eq:selmer-lower-bound}
\lambda\left(\Sel^{\hh}(\EC/K_\cyc)^\vee\right)  \geq 1.
\end{equation*}
Thus, condition (3$'$) implies that
\[
\lambda\left(\Sel^\hh(\EC/K_\cyc)^\vee\right) = 1 \quad \text{and} \quad \mu\left(\Sel^\hh(\EC/K_\cyc)^\vee\right) =0,
\]
which allows us to apply Corollary \ref{cor:sc-satisfied} and verify hypothesis $\eqref{hyp:SC}$.

\begin{example}Let $\EC/\Q$ be the elliptic curve with Cremona label \href{https://www.lmfdb.org/EllipticCurve/Q/91a1/}{\texttt{91a1}}, let $K=\Q(\sqrt{-11})$, and let $p=3$.
Note that $\rank_\Z \EC(\Q)=1$.
As before, one can use Magma to easily check conditions (0$'$) and (1).
Note that $91=7 \cdot 13$, and both $7$ and $13$ are inert in $K$, so \eqref{hyp: GHH} holds -- this verifies (2).
The twist $\EC^{(K)}$ has Cremona label \href{https://www.lmfdb.org/EllipticCurve/Q/11011r1/}{\texttt{11011r1}}, and from LMFDB we obtain
\[
\lambda(L_3^+(\EC))=1\qquad\text{and}\qquad \lambda(L_3^+(\EC^{(K)}))=0,
\]
and
\[
\mu(L_3^+(\EC))=\mu(L_3^+(\EC^{(K)}))=0
\]
which verifies condition (3$'$).
It is easy to check that $3$ does not divide the Tamagawa number of $\EC/K$.
This verifies condition (5').
Therefore, Hypothesis~\eqref{hyp:SC} holds for the tuple $(\texttt{91a1},\Q(\sqrt{-11}),3)$.
\end{example}

\bibliographystyle{amsalpha}
\bibliography{references}

\providecommand{\bysame}{\leavevmode\hbox to3em{\hrulefill}\thinspace}
\providecommand{\MR}{\relax\ifhmode\unskip\space\fi MR }
\providecommand{\MRhref}[2]{%
  \href{http://www.ams.org/mathscinet-getitem?mr=#1}{#2}
}
\providecommand{\href}[2]{#2}
\begin{thebibliography}{{LMF}23}

\bibitem[BCP97]{Magma}
Wieb Bosma, John Cannon, and Catherine Playoust, \emph{The {M}agma algebra
  system. {I}. {T}he user language}, J. Symbolic Comput. \textbf{24} (1997),
  no.~3-4, 235--265, Computational algebra and number theory (London, 1993).
  \MR{MR1484478}

\bibitem[Ber95]{Bertolini-Compositio}
Massimo Bertolini, \emph{Selmer groups and {H}eegner points in anticyclotomic
  {$\mathbf Z_p$}-extensions}, Compositio Math. \textbf{99} (1995), no.~2,
  153--182.

\bibitem[BH97]{BH97}
Paul~N Balister and Susan Howson, \emph{Note on {N}akayama’s lemma for
  compact ${\Lambda}$-modules}, Asian J. Math. \textbf{1} (1997), no.~2,
  224--229.

\bibitem[Bou65]{bourbaki}
N.~Bourbaki, \emph{\'{E}l\'{e}ments de math\'{e}matique. {F}asc. {XXXI}.
  {A}lg\`ebre commutative. {C}hapitre 7: {D}iviseurs}, Actualit\'{e}s
  Scientifiques et Industrielles [Current Scientific and Industrial Topics],
  vol. No. 1314, Hermann, Paris, 1965. \MR{260715}

\bibitem[CG96]{CG96}
John Coates and Ralph Greenberg, \emph{Kummer theory for abelian varieties over
  local elds}, Invent. math \textbf{124} (1996), 129--174.

\bibitem[Chi02]{chi02}
Gautam Chinta, \emph{Analytic ranks of elliptic curves over cyclotomic fields},
  J. Reine Angew. Math. \textbf{2002} (2002), no.~544, 13--24.

\bibitem[CW21]{castellawan}
Francesc Castella and Xin Wan, \emph{{Perrin-Riou's main conjecture for
  elliptic curves at supersingular primes}}, Math. Ann. \textbf{389} (2021),
  no.~3, 2595--2636.

\bibitem[DR21]{dionray}
C\'{e}dric Dion and Jishnu Ray, \emph{On the {M}ordell-{W}eil ranks of
  supersingular abelian varieties over $\mathbb{Z}_p^2$-extensions}, 2021, to
  appear in Israel J. Math., available at arXiv:2112.00280.

\bibitem[Gre99a]{Gre_PCMS}
Ralph Greenberg, \emph{Introduction to {I}wasawa theory for elliptic curves},
  Arithmetic algebraic geometry. IAS/Park City Math. Ser \textbf{9} (1999),
  407--464.

\bibitem[Gre99b]{greenberg99}
\bysame, \emph{Iwasawa theory for elliptic curves}, Arithmetic theory of
  elliptic curves ({C}etraro, 1997), Lecture Notes in Math., vol. 1716,
  Springer, Berlin, 1999, pp.~51--144. \MR{1754686}

\bibitem[Gre06]{greenberg06}
\bysame, \emph{On the structure of certain {G}alois cohomology groups}, Doc.
  Math. (2006), no.~Extra Vol., 335--391.

\bibitem[Gre10]{greenberg10}
\bysame, \emph{Surjectivity of the global-to-local map defining a {S}elmer
  group}, Kyoto J. Math. \textbf{50} (2010), no.~4, 853--888.

\bibitem[Gre16]{Greenberg16}
\bysame, \emph{On the structure of {S}elmer groups}, pp.~225--252, Springer
  International Publishing, Switzerland, 2016.

\bibitem[Ham22]{parham}
Parham Hamidi, \emph{Residual supersingular {I}wasawa theory over quadratic
  imaginary fields}, 2022, preprint, arXiv:2206.03679.

\bibitem[Hid22]{hida2022_book}
Haruzo Hida, \emph{Elementary modular {I}wasawa theory}, Series on Number
  Theory and its Applications, vol.~16, World Scientific Publishing Co. Pte.
  Ltd., Hackensack, NJ, 2022.

\bibitem[HL20]{hunglim}
Pin-Chi Hung and Meng~Fai Lim, \emph{On the growth of {M}ordell-{W}eil ranks in
  {$p$}-adic {L}ie extensions}, Asian J. Math. \textbf{24} (2020), no.~4,
  549--570.

\bibitem[HO10]{HO10}
Yoshitaka Hachimori and Tadashi Ochiai, \emph{Notes on non-commutative
  {I}wasawa theory}, Asian J. Math. \textbf{14} (2010), no.~1, 11--18.

\bibitem[Kat04]{kato04}
Kazuya Kato, \emph{{$p$}-adic {H}odge theory and values of zeta functions of
  modular forms}, Ast\'erisque (2004), no.~295, ix, 117--290, Cohomologies
  $p$-adiques et applications arithm{\'e}tiques. III.

\bibitem[Kim07]{kim_parity}
Byoung Du (B.~D.) Kim, \emph{The parity conjecture for elliptic curves at
  supersingular reduction primes}, Compositio Mathematica \textbf{143} (2007),
  no.~1, 47–72.

\bibitem[Kim13]{kim13}
Byoung~Du Kim, \emph{The plus/minus {S}elmer groups for supersingular primes},
  J. Aust. Math. Soc. \textbf{95} (2013), no.~2, 189--200.

\bibitem[Kim14]{kim14}
\bysame, \emph{Signed-{S}elmer groups over the {$\mathbb{Z}_p^2$}-extension of
  an imaginary quadratic field}, Canad. J. Math. \textbf{66} (2014), no.~4,
  826--843.

\bibitem[KLR22]{KLR}
Debanjana Kundu, Antonio Lei, and Anwesh Ray, \emph{Arithmetic statistics and
  noncommutative {I}wasawa theory}, Doc. Math. \textbf{27} (2022), 89--149.

\bibitem[KMS23]{KMS}
S\"oren Kleine, Ahmed Matar, and Ramdorai Sujatha, \emph{{On the
  $\mathfrak{M}_H(G)$-property}}, 2023, to appear in Math. Proc. Camb. Philos.
  Soc.

\bibitem[Kob03]{kobayashi03}
Shinichi Kobayashi, \emph{Iwasawa theory for elliptic curves at supersingular
  primes}, Invent. Math. \textbf{152} (2003), no.~1, 1--36.

\bibitem[KR24]{KR2}
Debanjana Kundu and Anwesh Ray, \emph{Statistics for {I}wasawa invariants of
  elliptic curves, {II}}, International Journal of Number Theory \textbf{20}
  (2024), no.~04, 1099--1124.

\bibitem[LL21]{lei2021akashi}
Antonio Lei and Meng~Fai Lim, \emph{Akashi series and euler characteristics of
  signed selmer groups of elliptic curves with semistable reduction at primes
  above $ p$}, Journal de th{\'e}orie des nombres de Bordeaux \textbf{33}
  (2021), no.~3.2, 997--1019.

\bibitem[{LMF}23]{lmfdb}
The {LMFDB Collaboration}, \emph{The {L}-functions and modular forms database},
  \url{https://www.lmfdb.org}, 2023, [Online; accessed 4 December 2023].

\bibitem[LP19]{LP}
Antonio Lei and Bharathwaj Palvannan, \emph{Codimension two cycles in {I}wasawa
  theory and elliptic curves with supersingular reduction}, Forum Math. Sigma
  \textbf{7} (2019), Paper No. e25, 81.

\bibitem[LS20a]{leisprung}
Antonio Lei and Florian Sprung, \emph{Ranks of elliptic curves over
  {$\mathbb{Z}_p^2$}-extensions}, Israel J. Math. \textbf{236} (2020), no.~1,
  183--206.

\bibitem[LS20b]{leisujatha}
Antonio Lei and Ramdorai Sujatha, \emph{On {S}elmer groups in the supersingular
  reduction case}, Tokyo J. Math. \textbf{43} (2020), no.~2, 455--479.

\bibitem[LV19]{LongoVigni2}
Matteo Longo and Stefano Vigni, \emph{Plus/minus {H}eegner points and {I}wasawa
  theory of elliptic curves at supersingular primes}, Boll. Unione Mat. Ital.
  \textbf{12} (2019), no.~3, 315--347.

\bibitem[Mat21]{matar21}
Ahmed Matar, \emph{Kolyvagin's work and anticyclotomic tower fields: the
  supersingular case}, Acta Arith. \textbf{201} (2021), no.~2, 131--147.

\bibitem[Maz72]{mazur72}
Barry Mazur, \emph{Rational points of abelian varieties with values in towers
  of number fields}, Invent. Math. \textbf{18} (1972), 183--266.

\bibitem[Maz83]{Maz84}
\bysame, \emph{Modular curves and arithmetic}, Proceedings of the International
  Congress of Mathematicians, vol.~1, Warszawa, 1983, pp.~185--211.

\bibitem[MN19]{matarnekovar}
Ahmed Matar and Jan Nekov\'{a}\v{r}, \emph{Kolyvagin's result on the vanishing
  of {${\rm III}(E/K)[p^\infty]$} and its consequences for anticyclotomic
  {I}wasawa theory}, J. Th\'{e}or. Nombres Bordeaux \textbf{31} (2019), no.~2,
  455--501.

\bibitem[Nek13]{nekovar}
Jan Nekov\'{a}\u{r}, \emph{Some consequences of a formula of {M}azur and
  {R}ubin for arithmetic local constants}, Algebra Number Theory \textbf{7}
  (2013), no.~5, 1101--1120.

\bibitem[OV03]{OV03}
Yoshihiro Ochi and Otmar Venjakob, \emph{On the ranks of {I}wasawa modules over
  {$p$}-adic {L}ie extensions}, Math. Proc. Cambridge Philos. Soc. \textbf{135}
  (2003), no.~1, 25--43.

\bibitem[Pol03]{pollack03}
Robert Pollack, \emph{On the {$p$}-adic {$L$}-function of a modular form at a
  supersingular prime}, Duke Math. J. \textbf{118} (2003), no.~3, 523--558.

\bibitem[Roh88]{rohrlich88}
David~E. Rohrlich, \emph{{$L$}-functions and division towers}, Math. Ann.
  \textbf{281} (1988), no.~4, 611--632.

\bibitem[SU14]{SkinnerUrban}
Christopher Skinner and Eric Urban, \emph{The {I}wasawa main conjectures for
  {$\rm GL_2$}}, Invent. Math. \textbf{195} (2014), no.~1, 1--277. \MR{3148103}

\bibitem[Vat03]{vatsal}
Vinayak Vatsal, \emph{Special values of anticyclotomic {$L$}-functions}, Duke
  Math. J. \textbf{116} (2003), no.~2, 219--261.

\end{thebibliography}

\end{document}